\documentclass[11pt]{article}
\usepackage[margin=1.04in]{geometry}
\usepackage[T1]{fontenc}
\usepackage{lmodern}
\usepackage{microtype,amsmath,amssymb,amsthm,mathtools,bm}
\usepackage{booktabs,array,longtable}
\usepackage[round,authoryear]{natbib}
\usepackage{enumitem}
\usepackage{setspace,needspace,etoolbox}
\BeforeBeginEnvironment{theorem}{\Needspace{7\baselineskip}}
\BeforeBeginEnvironment{proposition}{\Needspace{7\baselineskip}}
\BeforeBeginEnvironment{corollary}{\Needspace{6\baselineskip}}
\BeforeBeginEnvironment{lemma}{\Needspace{6\baselineskip}}
\usepackage[colorlinks=true,linkcolor=black,citecolor=black,urlcolor=black]{hyperref}
\usepackage{bookmark}
\hypersetup{pdftitle={Adaptive Bayesian Estimation of Conditional Densities with Orthogonal Polynomials},pdfauthor={Research manuscript},pdfsubject={Conditional distributions, adaptive estimation, prediction, and polynomial term selection}}
\numberwithin{equation}{section}
\newtheorem{theorem}{Theorem}[section]
\newtheorem{proposition}[theorem]{Proposition}
\newtheorem{lemma}[theorem]{Lemma}
\newtheorem{corollary}[theorem]{Corollary}
\theoremstyle{definition}
\newtheorem{assumption}{Assumption}

\theoremstyle{remark}
\newtheorem{remark}[theorem]{Remark}
\newcommand{\R}{\mathbb R}
\newcommand{\N}{\mathbb N}
\newcommand{\E}{\mathbb E}
\newcommand{\Prob}{\mathbb P}
\newcommand{\ind}{\mathbf 1}

\newcommand{\norm}[1]{\left\lVert #1\right\rVert}

\newcommand{\esssup}{\operatorname*{ess\,sup}}

\newcommand{\TV}{\operatorname{TV}}
\newcommand{\KL}{\operatorname{KL}}

\newcommand{\cF}{\mathcal F}
\newcommand{\cA}{\mathcal A}
\newcommand{\cS}{\mathcal S}
\newcommand{\cD}{\mathcal D}
\newcommand{\eps}{\varepsilon}

\newcommand{\Gclass}{\mathcal G}

\setlist{nosep,leftmargin=1.6em}
\allowdisplaybreaks[2]
\title{\bfseries Adaptive Bayesian Estimation of Conditional Densities\\ Based on Orthogonal Polynomials}
\author{
Andriy Norets\textsuperscript{1},
Marco Stenborg Petterson\textsuperscript{2}
\\[0.7em]
{\small \textsuperscript{1}Brown University}\\
{\small \textsuperscript{2}University of Naples Federico II and CSEF}
}
\date{28 October 2026}
\begin{document}
\maketitle
\vspace{-1.4em}

\begin{abstract}
We study Bayesian estimation of conditional densities represented by normalized squares of orthogonal polynomial series. We specify the conditional density without modeling the marginal distribution of covariates. A prior on smoothness, the number of terms, and the coefficients gives adaptive posterior contraction at the minimax rate without a logarithmic factor when the normalized square-root representation of the true density is bounded away from zero and sufficiently smooth. For one outcome and $d_x$ covariates, the rate over classes with degree smoothness $s$ is $n^{-s/(2s+d_x+1)}$. The posterior predictive density attains the corresponding minimax squared Hellinger risk. For bounded nonnegative representations in Legendre and Chebyshev classes, we establish the minimax rate at every $s>0$. A prior that also permits a small mixture with the reference density adapts over these classes up to a logarithmic factor and retains the exact rate on the positive, sufficiently smooth subclass. Additional results concern identification, conditional distribution functions and quantiles, coefficient recovery and uniform convergence, independent Gaussian coefficient priors, and selection of polynomial terms.
\end{abstract}
\noindent\textit{Keywords:} conditional density estimation; Bayesian nonparametrics; posterior contraction; adaptive estimation; orthogonal polynomials; minimax risk.\\
\textit{JEL classification:} C11, C13, C14.
\clearpage

\section{Introduction}
Conditional distributions describe variation in the location, dispersion, and shape of an outcome distribution across covariate values. A large literature develops nonparametric methods for estimating conditional densities \citep{FanYaoTong1996,HallRacineLi2004,NoretsPati2017}. In this paper we develop an adaptive Bayesian procedure for conditional density estimation using orthogonal polynomial series. Under positivity and smoothness conditions, we establish minimax-optimal posterior contraction rates and show that the posterior predictive density attains the corresponding minimax risk.

Let $Y$ be an outcome and $X$ a vector of covariates. For a polynomial series $h(y,x)$ and an outcome reference probability measure $\mu_y$, consider
\begin{equation}\label{eq:intro-model}
 p_h(y\mid x)=\frac{h(y,x)^2}{\int h(u,x)^2\,\mu_y(du)},
\end{equation}
where squaring makes the density nonnegative. For finite expansions in the tensor bases used below, orthogonality in the outcome argument gives an explicit expression for the denominator. The conditional likelihood can therefore be evaluated without numerical integration over outcomes or estimation of the covariate density.

Multiplying $h$ by a nonzero function of $x$ leaves $p_h$ unchanged. We impose $\int h(y,x)\,\mu_y(dy)=1$ for almost every $x$, which also bounds the denominator below by one. Every conditional density $p$ has a unique nonnegative representation satisfying this normalization, obtained by dividing $\sqrt p$ by its integral with respect to $\mu_y$. Our smoothness assumptions impose decay restrictions on the orthogonal-series coefficients of this normalized square root, with the reference measures and basis held fixed.

The main prior we consider in this paper is hierarchical in a smoothness index $\alpha$, the number of terms, and the coefficients. Conditional on smoothness and dimension, we start with independent centered Gaussian coefficients with variances proportional to $j^{-2\alpha-1}$ and condition their joint distribution on a coefficient ellipsoid. The grid of smoothness indices and its probabilities depend on sample size but not on the observations themselves. If the true representation is bounded away from zero and sufficiently smooth, the posterior contracts at rate $n^{-\alpha/(2\alpha+1)}$ in Hellinger distance averaged over covariates. The bound is uniform over the stated smoothness, truth, and covariate-distribution classes. An exponential bound on the expected posterior probability outside a ball of this radius yields a uniform squared Hellinger risk bound of order $n^{-2\alpha/(2\alpha+1)}$ for the posterior predictive density. A matching minimax lower bound establishes optimality without a logarithmic loss.

For polynomial terms ordered by total degree in $q=d_x+1$ variables, degree smoothness $s$ corresponds to $\alpha=s/q$. The exact adaptive rate is therefore $n^{-s/(2s+q)}$. We can then apply these results to specific families of orthogonal polynomials, e.g. Legendre and Chebyshev polynomials. Specifically, the result applies for $s>q$ with Legendre polynomials and $s>q/2$ with Chebyshev polynomials. These restrictions give uniform convergence and boundedness over the respective coefficient ellipsoids.

Separately, we establish the minimax squared Hellinger rate $n^{-2s/(2s+q)}$ over uniformly bounded nonnegative representations at every $s>0$. The upper bound uses a testing-based estimator that may depend on smoothness. Our adaptive Bayesian extension to this broader class attains the benchmark up to a logarithmic factor.

We also provide additional results concerning conditional functionals, coefficient recovery, and uniform convergence under further assumptions. We also study priors with independent Gaussian coefficients and priors selecting subsets of polynomial terms, obtaining an oracle contraction bound for the latter.

The specification is related to the seminonparametric models of \citet{GT1989}, which combine a reference density with a squared polynomial.

The proofs build on the posterior-convergence and model-averaging arguments of \citet{GGV2000} and \citet{GLV2008}. \citet{Scricciolo2006} obtains exact rates for exponential Fourier models with smoothness-matched Gaussian priors and adaptation over a finite set of smoothness levels. \citet{GaoZhou2016} obtain exact adaptation with modified block priors independent of smoothness and sample size. The adaptive prior studied here depends on sample size. \citet{ShenGhosal2015} develop general results for random-series priors. For conditional densities, \citet{NoretsPati2017} establish adaptive rates up to logarithmic factors for covariate-dependent normal mixtures. \citet{LuoLuo2026} study univariate weighted orthogonal-polynomial densities with smoothness-dependent priors and a modification permitting densities to approach zero. Our results concern unknown smoothness in a conditional model and include matching conditional minimax bounds.

Sections~\ref{sec:model} and \ref{sec:smoothness} introduce the model and smoothness conditions. Section~\ref{sec:adaptive} gives the main adaptive and minimax results. Sections~\ref{sec:functionals} and \ref{sec:polynomials} consider functionals, uniform convergence, and polynomial degree. Section~\ref{sec:floor} treats densities that may vanish. Sections~\ref{sec:ordinary} and \ref{sec:sparse} study independent Gaussian priors and selection of terms. The appendices contain the proofs.

\section{Model and notation}\label{sec:model}
\subsection{Data generating process}

We observe independent and identically distributed data $\cD_n=(Z_1,\ldots,Z_n)$, where $Z_i=(Y_i,X_i)$ consists of an outcome and covariates with joint distribution
\[
 P_0(dy,dx)=p_0(y\mid x)\,\mu_y(dy)G_0(dx).
\]
Here $\mu_y$ and $\mu_x$ are fixed probability reference measures on standard Borel spaces $\mathcal Y$ and $\mathcal X$, and $G_0\ll\mu_x$. Write $\mu=\mu_y\otimes\mu_x$ and $z=(y,x)$. Throughout, conditional densities are measurable, nonnegative, and integrate to one with respect to $\mu_y$ for $\mu_x$-almost every $x$. Let $P_{p,G}$ denote the joint distribution corresponding to a conditional density $p$ and a covariate distribution $G\ll\mu_x$. We write $P_0^n=P_0^{\otimes n}$ for the distribution of the sample and $\E_0$ for expectation under the true distribution.

For conditional densities $p,q$ and $G\ll\mu_x$, define
\begin{align}
 d_G^2(p,q)
 &=\int(\sqrt p-\sqrt q)^2\,d\mu_y\,dG,
 \label{eq:hell}\\
 K_G(p,q)
 &=\int p\log(p/q)\,d\mu_y\,dG,
 \label{eq:kl}\\
 V_G(p,q)
 &=\int p\{\log(p/q)\}^2\,d\mu_y\,dG.
 \label{eq:var}
\end{align}
We measure estimation error by the integrated squared Hellinger distance $d_{G_0}^2$, which takes values in $[0,2]$. The quantities $K_G$ and $V_G$ are the Kullback--Leibler divergence and the uncentered second moment of the log likelihood ratio.

\begin{assumption}[Covariate distribution]\label{ass:design}
The density $g_0=dG_0/d\mu_x$ satisfies
\[
 0<\underline g\le g_0(x)\le\overline g<\infty
 \qquad\text{for }\mu_x\text{-almost every }x.
\]
\end{assumption}

Let $\Gclass(\underline g,\overline g)$ denote the class of covariate distributions satisfying Assumption~\ref{ass:design}. We do not impose any smoothness requirement on $g_0$.

We use $\|\cdot\|_2$ for the $L_2(\mu)$ norm of functions and the Euclidean or $\ell_2$ norm of coefficients, and $\|\cdot\|_\infty$ for the supremum norm of coefficients and the essential supremum norm of functions. Let $\N_0=\{0,1,\ldots\}$ and let $N(\eps,A,d)$ denote the covering number of $A$ by $d$-balls of radius $\eps$. Positive finite constants $c,C$ may change from line to line. For nonnegative sequences, $a_n\lesssim b_n$ means $a_n\le Cb_n$ for all sufficiently large $n$, with $C$ independent of $n$; $a_n\asymp b_n$ means $a_n\lesssim b_n\lesssim a_n$. Unless stated otherwise, equalities and inequalities between functions hold $\mu$-almost everywhere, or $\mu_x$-almost everywhere for functions of $x$. If $\mu_y(dy)=w_y(y)\,dy$, the corresponding Lebesgue density is $f(y\mid x)=w_y(y)p(y\mid x)$. Expressing a fixed conditional distribution relative to Lebesgue measure rather than $\mu_y$ leaves likelihood ratios and Hellinger loss unchanged. The smoothness classes, however, are defined through the coefficients of the normalized square-root representation relative to the chosen reference measures and ordered basis.

\subsection{Series model}

Fix an ordered real orthonormal basis $\{\phi_j:j\ge1\}$  of
\[
 \mathcal H_0
 =\left\{
 v\in L_2(\mu):
 \int v(y,x)\,\mu_y(dy)=0
 \text{ for }\mu_x\text{-almost every }x
 \right\},
\]

For $\theta\in\ell_2$, define
\begin{equation}\label{eq:series-model}
 h_\theta=1+\sum_{j\ge1}\theta_j\phi_j,\qquad
 Q_\theta(x)=\int h_\theta(y,x)^2\,\mu_y(dy),\qquad
 p_\theta(y\mid x)=\frac{h_\theta(y,x)^2}{Q_\theta(x)}.
\end{equation}
By orthonormality, the series converges in $L_2(\mu)$, and $\|h_\theta-h_\eta\|_2=\|\theta-\eta\|_2$. Moreover,
\[
 \int h_\theta(y,x)\,\mu_y(dy)=1,\qquad
 1\le Q_\theta(x)<\infty
 \quad\text{for }\mu_x\text{-almost every }x.
\]

Setting $h_\theta(\cdot,x)=1$ on the exceptional set makes $p_\theta(\cdot\mid x)$ a density for every $x$. All coefficient priors considered below assign probability one to finitely supported sequences. The true representation, however, may have infinitely many nonzero coefficients.

For a measurable function $h$, denote $Q_h(x)=\int h(y,x)^2\,\mu_y(dy)$ and $p_h(y\mid x)=h(y,x)^2/Q_h(x)$ whenever $0<Q_h(x)<\infty$ almost everywhere.

\begin{proposition}[Normalization and Hellinger distance]
\label{prop:geometry} Let $h,k\in L_2(\mu)$ be real functions satisfying
\[
 \int h(y,x)\,\mu_y(dy)
 =\int k(y,x)\,\mu_y(dy)=1
 \quad\text{for }\mu_x\text{-almost every }x.
\] 
Then $1\le Q_h,Q_k<\infty$ almost everywhere, and $p_h,p_k$ are conditional densities. If $g_0\le\overline g$, then
\begin{equation}\label{eq:global-H}
 d_{G_0}(p_h,p_k)
 \le\sqrt{\overline g}\,\|h-k\|_2.
\end{equation}
If also $0\le h,k\le B$ for some $B\ge1$ and $g_0\ge\underline g>0$, then
\begin{equation}\label{eq:inverse-H}
 \|h-k\|_2
 \le\sqrt{(B+B^2)/\underline g}\,d_{G_0}(p_h,p_k).
\end{equation}

Every conditional density $p$ has a unique nonnegative measurable representation satisfying $\int h_p(y,x)\,\mu_y(dy)=1$ for $\mu_x$-almost every $x$, up to $\mu$-null sets:
\begin{equation}\label{eq:canonical}
 c_p(x)=\int\sqrt{p(y\mid x)}\,\mu_y(dy),\qquad
 h_p(y,x)=\frac{\sqrt{p(y\mid x)}}{c_p(x)},\qquad
 p_{h_p}=p.
\end{equation}
Here $0<c_p\le1$ almost everywhere, and $0\le m\le p\le M<\infty$ implies $\sqrt m\le h_p\le M$.
\end{proposition}

By orthonormality, \eqref{eq:global-H} gives
\[
 d_{G_0}(p_\theta,p_\eta)
 \le\sqrt{\overline g}\,\|\theta-\eta\|_2,
\]
without requiring nonnegative $h_\theta,h_\eta$. For bounded nonnegative representations, \eqref{eq:inverse-H} gives the reverse bound. We call $h_p$ the normalized square-root representation of $p$. Although $Q_{h_p}(x)=c_p(x)^{-2}<\infty$ almost everywhere, $h_p$ need not belong to $L_2(\mu)$. The smoothness conditions below impose this additional requirement.

For $\mathcal Y\subseteq\R$ and $\mu_y$ with infinite support, let $\psi_0=1,\psi_1,\ldots$ be orthonormal polynomials in $L_2(\mu_y)$ and let $\{\chi_\ell\}$ be an orthonormal basis of $L_2(\mu_x)$. Set $\phi_j(y,x)=\psi_{k_j}(y)\chi_{\ell_j}(x)$, $k_j\ge1$. For a finite expansion, $a_k(x)=\sum_{j:k_j=k}\theta_j\chi_{\ell_j}(x)$ gives
\begin{equation}\label{eq:normalizer}
 h_\theta(y,x)=1+\sum_{k\ge1}a_k(x)\psi_k(y),\qquad
 Q_\theta(x)=1+\sum_{k\ge1}a_k(x)^2.
\end{equation}

\begin{proposition}[Identification]\label{prop:identification}
Suppose $\mathcal Y\subseteq\R$ and $\mu_y$ has infinite support. Let $h,k$ be measurable functions whose sections $h(\cdot,x)$ and $k(\cdot,x)$ are real polynomials in $L_2(\mu_y)$ with integral one for $\mu_x$-almost every $x$. Then $p_h=p_k$ $\mu$-almost everywhere implies $h=k$ $\mu$-almost everywhere. Consequently, two finite expansions in the same orthonormal tensor basis with polynomial outcome functions generate the same conditional density only if their coefficient vectors agree after padding with zeros.
\end{proposition}

Within this finite polynomial model, the coefficient sequence is identified; a truncation index allowing trailing zeros need not be. Under Assumption~\ref{ass:design}, equality of the observed joint distributions implies equality of conditional densities $\mu$-almost everywhere. Proofs of both propositions are given in Appendix~\ref{app:model-details}.

The normalization $\int h_\theta(y,x)\,\mu_y(dy)=1$ excludes terms depending only on $x$. Fixing only the overall constant coefficient does not identify the representation. For example, under uniform reference measures on $[-1,1]^2$, $(1+ax)(1+by)$ and $1+by$ have constant coefficient one and generate the same conditional density when $0<|a|<1$, but only the second satisfies the normalization. This restriction also changes the finite polynomial model: dividing an unrestricted polynomial by its nonzero integral with respect to $y$ generally produces a rational function of $x$, which need not have a finite polynomial expansion.

\section{Smoothness and approximation}\label{sec:smoothness}

For $\alpha>0$ and $R\ge0$, define
\begin{equation}\label{eq:ellipsoid}
 \|\theta\|_{2,\alpha}^2
 =\sum_{j\ge1}j^{2\alpha}\theta_j^2,
 \qquad
 \Theta_\alpha(R)
 =\{\theta\in\ell_2:\|\theta\|_{2,\alpha}\le R\}.
\end{equation}
The radius $R$ bounds the weighted coefficient norm, while a larger $\alpha$ imposes a stronger restriction on the higher-index coefficients. This is a smoothness condition on $h_\theta$, relative to the chosen basis and ordering. Section~\ref{sec:polynomials} relates this condition to polynomial degree and differentiability.

Let $\theta^{(J)}=(\theta_1,\ldots,\theta_J,0,\ldots)$ denote the sequence obtained by retaining the first $J$ coefficients. Define the corresponding finite-dimensional ellipsoid by
\[
 \Theta_{\alpha,J}(R)
 =
 \left\{
 \vartheta\in\R^J:
 \sum_{j=1}^J j^{2\alpha}\vartheta_j^2\le R^2
 \right\}.
\]
We regard finite vectors as infinite sequences by appending zeros.

\begin{assumption}[Basis growth]\label{ass:kernel}
For some $C_K<\infty$ and $\rho\ge1$,
\begin{equation}\label{eq:kernel}
 \esssup_z\sum_{j=1}^m\phi_j(z)^2
 \le C_Km^\rho,
 \qquad m\ge1.
\end{equation}
\end{assumption}

By Cauchy--Schwarz, \eqref{eq:kernel} implies
\[
 \left\|\sum_{j\le m}\theta_j\phi_j\right\|_\infty
 \le\sqrt{C_K}\,m^{\rho/2}\|\theta^{(m)}\|_2.
\]
Thus $\rho$ controls how the supremum-norm bound grows with the number of terms, for a given coefficient norm. Orthonormality gives $\int\sum_{j=1}^m\phi_j(z)^2\,\mu(dz)=m$, so \eqref{eq:kernel} requires $\rho\ge1$.

For the main adaptive result, fix
\begin{equation}\label{eq:range}
 \rho/2<\alpha_-<\alpha_+<\infty,\qquad
 0<R_0<R_*<\infty,\qquad
 0<b_0<1.
\end{equation}
The interval $[\alpha_-,\alpha_+]$ is the range of smoothness indices over which adaptation is sought. For each $\alpha$ in this interval, define
\begin{equation}\label{eq:trueclass}
 \cF_\alpha(R_0,b_0)
 =
 \{p_\theta:
   \theta\in\Theta_\alpha(R_0),\
   h_\theta\ge b_0\ \mu\text{-almost everywhere}\}.
\end{equation}
The true density is assumed to belong to $\cF_{\alpha_0}(R_0,b_0)$ for some unknown $\alpha_0\in[\alpha_-,\alpha_+]$. The constant $R_*$ is the ellipsoid radius used in the prior below. The inequality $R_0<R_*$ leaves a positive margin in the weighted coefficient norm between truncated true coefficients and the boundary of the prior ellipsoid.

On these classes, nonnegativity and normalization imply $h_\theta=h_{p_\theta}$. Assumption~\ref{ass:kernel} and $\alpha_->\rho/2$ also give $\|h_\theta\|_\infty\le1+CR_0$, with a constant common to all $\alpha\in[\alpha_-,\alpha_+]$ and $\theta\in\Theta_\alpha(R_0)$.

For $\theta\in\Theta_\alpha(R_0)$, truncation gives
\[
 \|\theta-\theta^{(J)}\|_2
 \le R_0(J+1)^{-\alpha}.
\]
If $g_0\le\overline g$, Proposition~\ref{prop:geometry} therefore implies
\[
 d_{G_0}(p_\theta,p_{\theta^{(J)}})
 \le\sqrt{\overline g}\,R_0(J+1)^{-\alpha}.
\]
This Hellinger approximation bound holds for every $\alpha>0$, without a positivity restriction on $h_\theta$ or a basis-growth assumption.

Under Assumption~\ref{ass:kernel}, the stronger restriction $\alpha\ge\alpha_->\rho/2$ also gives
\[
 \|h_\theta-h_{\theta^{(J)}}\|_\infty
 \le CR_0J^{\rho/2-\alpha}.
\]
Consequently, if $h_\theta\ge b_0$, then $h_{\theta^{(J)}}\ge b_0/2$ for all sufficiently large $J$, uniformly over the stated smoothness classes. The uniform bounds follow from Lemma~\ref{lem:interpolation}.

\section{Adaptive estimation}\label{sec:adaptive}
\subsection{Prior}

We specify a prior on a smoothness index $\alpha$, a truncation level $J$, and the series coefficients $\theta$. The prior depends on sample size, but not on the realized observations, the covariate distribution $G_0$, or the unknown smoothness index $\alpha_0$.

For $n\ge3$ and $\alpha\in[\alpha_-,\alpha_+]$, define
\begin{equation}\label{eq:rates}
 t_n(\alpha)=n^{1/(2\alpha+1)},\qquad
 \eps_n(\alpha)=n^{-\alpha/(2\alpha+1)},\qquad
 J_n(\alpha)=\lceil t_n(\alpha)\rceil.
\end{equation}
Here $t_n(\alpha)$ is the benchmark dimension at which the approximation order $J^{-\alpha}$ equals $\sqrt{J/n}$, $\eps_n(\alpha)$ is the corresponding target Hellinger contraction rate, and $J_n(\alpha)$ is the integer version of the benchmark dimension. These quantities satisfy $n\eps_n(\alpha)^2=t_n(\alpha)$. The truncation level $J$ is random under the prior; it is not fixed at $J_n(\alpha)$.

Fix $c_\Delta,A,\lambda,\tau>0$. Define the smoothness grid by
\begin{equation}\label{eq:grid}
 \cA_n
 =
 \left\{
 \alpha_-+\frac{kc_\Delta}{\log n}:
 k=0,\ldots,
 \left\lfloor
 \frac{(\alpha_+-\alpha_-)\log n}{c_\Delta}
 \right\rfloor
 \right\},
 \qquad
 L_n=|\cA_n|=O(\log n).
\end{equation}
The spacing $c_\Delta/\log n$ ensures that replacing any $\alpha_0\in[\alpha_-,\alpha_+]$ by the largest grid point not exceeding it changes the benchmark dimension and target rate by at most constant factors.

Draw $\alpha\in\cA_n$ and $J\in\N_0$ independently with probabilities
\begin{align}
 w_n(\alpha)
 &=
 \frac{e^{-At_n(\alpha)}}
 {\sum_{\beta\in\cA_n}e^{-At_n(\beta)}},
 &&\alpha\in\cA_n,
 \label{eq:weights}\\
 \pi(J)
 &=(1-e^{-\lambda})e^{-\lambda J},
 &&J=0,1,\ldots.
 \label{eq:Jprior}
\end{align}
For $J\ge1$, conditional on $(\alpha,J)$, draw $\theta$ from
\begin{equation}\label{eq:coeffprior}
 \widetilde\Pi_{\alpha,J}
 =
 \bigotimes_{j=1}^J N(0,\tau^2j^{-2\alpha-1}),
 \qquad
 \Pi_{\alpha,J}
 =
 \widetilde\Pi_{\alpha,J}
 \bigl(\,\cdot\mid\Theta_{\alpha,J}(R_*)\bigr).
\end{equation}
For $J=0$, set $\theta=0$. In every case, coefficients above $J$ are zero. The resulting coefficient prior is
\begin{equation}\label{eq:fullprior}
 \Pi_n
 =
 \sum_{\alpha\in\cA_n}w_n(\alpha)
 \sum_{J\ge0}\pi(J)\Pi_{\alpha,J}.
\end{equation}
We also use $\Pi_n$ for the joint prior on $(\alpha,J,\theta)$ and for the induced prior on densities; the argument specifies which is meant.

The weights $w_n$ penalize lower smoothness indices through their larger benchmark dimensions $t_n(\alpha)$. The geometric prior $\pi(J)$ separately penalizes large truncation levels.

Conditioning on $\Theta_{\alpha,J}(R_*)$ in \eqref{eq:coeffprior} enforces $\|\theta\|_{2,\alpha}\le R_*$ for every $J$. Under the unconditioned Gaussian distribution,
\[
 \E_{\widetilde\Pi_{\alpha,J}}\|\theta\|_{2,\alpha}^2
 =
 \tau^2\sum_{j=1}^J j^{-1},
\]
which grows with $J$. Each finite ellipsoid has positive Gaussian probability, so the conditional distribution is well defined. The positivity condition in \eqref{eq:trueclass} is imposed on the truth class, not on the prior draws.

\begin{samepage}
\subsection{Posterior contraction and predictive risk}
The prior does not use the true smoothness index $\alpha_0$. The next result shows that the posterior nevertheless contracts at the rate $\eps_n(\alpha_0)$. It also bounds the posterior probability of truncation levels and smoothness components whose complexity exceeds the benchmark $t_n(\alpha_0)$.

\begin{theorem}[Adaptive contraction]\label{thm:adaptive}
Suppose $g_0\le\overline g$, Assumption~\ref{ass:kernel} holds, and the constants satisfy \eqref{eq:range}. For the prior \eqref{eq:fullprior}, there exist $M,L,c,C>0$ and $n_0$ such that
\begin{equation}\label{eq:mainbound}
 \begin{split}
 \E_0\Pi_n\bigl\{&d_{G_0}(p_\theta,p_0)>M\eps_n(\alpha_0)
 \ \text{or}\ \max\{J,t_n(\alpha)\}>Lt_n(\alpha_0)
 \mid\cD_n\bigr\}\\
 &\hspace{5em}\le C e^{-ct_n(\alpha_0)},\qquad n\ge n_0,
 \end{split}
\end{equation}
uniformly over $\alpha_0\in[\alpha_-,\alpha_+]$, $p_0\in\cF_{\alpha_0}(R_0,b_0)$, and covariate densities bounded above by $\overline g$. The result holds for every fixed $A>0$. The constants and $n_0$ may depend on the fixed quantities in \eqref{eq:range}, the basis bounds $(C_K,\rho)$, $\overline g$, and the prior parameters $(c_\Delta,A,\lambda,\tau)$.

\end{theorem}
\end{samepage}
The expectation in \eqref{eq:mainbound} is over repeated samples from the true distribution. The bound controls the posterior probability of either a large density error or excessive model complexity. Its constants are common to the stated smoothness, truth, and covariate classes. In particular, the same prior attains the rate corresponding to each $\alpha_0$ in the prescribed interval.

The complexity conclusion is an upper bound on $J$ and $t_n(\alpha)$; it does not assert that either is asymptotically equal to its benchmark. Nor does it identify a unique smoothness index: a density can belong to several of the nested smoothness classes.

For each fixed true distribution, the posterior probability in \eqref{eq:mainbound} also tends to zero almost surely along an infinite iid sample. Appendix~\ref{app:main} proves the theorem and derives this consequence from the summable expectation bound.

If $\alpha_0$ is known, the prior $\Pi^{(\alpha_0)}=\sum_{J\ge0}\pi(J)\Pi_{\alpha_0,J}$ is independent of sample size and satisfies the same contraction and dimension bounds. The truncation level remains random.

The posterior predictive density averages the conditional densities under the posterior. The expectation bound in Theorem~\ref{thm:adaptive} also gives a frequentist risk bound for this estimator.

\begin{corollary}[Predictive risk]\label{cor:risk}
Under Theorem~\ref{thm:adaptive}, let $\bar p_n(y\mid x)=\int p_\theta(y\mid x)\,\Pi_n(d\theta\mid\cD_n)$. Uniformly over the classes in that theorem,
\begin{equation}\label{eq:risk}
 \E_0 d_{G_0}^2(\bar p_n,p_0)
 \le \E_0\int d_{G_0}^2(p_\theta,p_0)\,\Pi_n(d\theta\mid\cD_n)
 \le C\eps_n(\alpha_0)^2.
\end{equation}
\end{corollary}
The first inequality follows from convexity of squared Hellinger loss in the estimated density. For the second, split the posterior integral according to whether the distance exceeds $M\eps_n(\alpha_0)$. Since squared Hellinger distance is at most two,
\[
 \E_0\int d_{G_0}^2(p_\theta,p_0)\,\Pi_n(d\theta\mid\cD_n)
 \le M^2\eps_n(\alpha_0)^2+2Ce^{-ct_n(\alpha_0)}.
\]
The exponential term is smaller than the squared contraction rate, uniformly over the stated classes. This gives the bound for both a posterior draw, averaged over the posterior, and the posterior predictive density. Appendix~\ref{app:predictive-details} gives the details.

\subsection{Minimax optimality}\label{sec:minimax}
The preceding corollary gives an upper bound on the risk of the posterior predictive density. We now show that no conditional-density estimator can attain a smaller order of worst-case risk over the same smoothness class, even when the estimator knows the smoothness index and the covariate distribution.

\begin{theorem}[Minimax lower bound]\label{thm:minimax}
Under Assumptions~\ref{ass:design}--\ref{ass:kernel} and \eqref{eq:range},
\begin{equation}\label{eq:minimax}
 \inf_{\widehat p_n}\sup_{p_0\in\cF_\alpha(R_0,b_0)}
 \E_{p_0,G_0}^n d_{G_0}^2(\widehat p_n,p_0)
 \ge c n^{-2\alpha/(2\alpha+1)},\qquad n\ge1,
\end{equation}
for a constant $c>0$ uniform over $\alpha\in[\alpha_-,\alpha_+]$ and $G_0\in\Gclass(\underline g,\overline g)$. The constant may depend on the fixed class bounds, smoothness endpoints, and basis-growth constants. The infimum is over measurable conditional-density estimators, including estimators that know $\alpha$ and $G_0$.
\end{theorem}
Together with Corollary~\ref{cor:risk}, \eqref{eq:minimax} establishes minimax squared Hellinger risk of order $n^{-2\alpha/(2\alpha+1)}$. The posterior predictive density attains this order without knowing $\alpha$, so adaptation introduces no logarithmic loss in this bound. Unlike the upper risk bound, this matching lower bound uses a positive lower bound on the covariate density. Appendix~\ref{app:minimax} proves it using perturbations of the reference representation $h=1$.

\section{Functionals and uniform convergence}\label{sec:functionals}
We next consider quantities obtained from the estimated conditional distribution and stronger notions of convergence. Integrated Hellinger contraction directly controls some conditional functionals. Recovering the series coefficients and obtaining uniform convergence require additional assumptions.

\subsection{Conditional distribution functions, moments, and quantiles}
For a real outcome, write
\[
 F_p(t\mid x)=\int_{(-\infty,t]}p(y\mid x)\,\mu_y(dy),
 \qquad
 m_{a,p}(x)=\int a(y,x)p(y\mid x)\,\mu_y(dy),
\]
where the first integral is restricted to $\mathcal Y$ and $a$ is a bounded measurable function. These quantities include conditional probabilities and, on a bounded outcome support, conditional means and second moments. The next proposition bounds their errors by the Hellinger distance at each covariate value. After integration over $G_0$, the density-risk bounds therefore give bounds on their average squared errors across covariates.

Quantiles require an additional condition: the true distribution function must increase sufficiently quickly near the quantiles of interest. For a bound uniform over all quantile indices, we impose a positive lower bound on the true Lebesgue density throughout the outcome interval.

\begin{proposition}[Conditional functionals]\label{prop:functionals}
Let $p,q$ be conditional densities and put $H_x^2=\int(\sqrt p-\sqrt q)^2\,d\mu_y$. For a real outcome, their conditional distribution functions satisfy
\begin{align}
 \sup_t|F_p(t\mid x)-F_q(t\mid x)|&\le H_x,\label{eq:cdf}\\
 \left|\int a(y,x)(p-q)(y\mid x)\,\mu_y(dy)\right|
 &\le2\|a\|_\infty H_x\label{eq:moment}
\end{align}
for any bounded measurable $a$, for almost every $x$. If $\mathcal Y=[y_-,y_+]$, the conditional distributions have Lebesgue densities, and the true density $f_0$ is bounded below by $m>0$ on $[y_-,y_+]\times\mathcal X$, then the generalized conditional quantiles of $p$ and $p_0$ satisfy
\begin{equation}\label{eq:quantile}
 \sup_{0<\tau<1}|Q_p(\tau\mid x)-Q_0(\tau\mid x)|
 \le m^{-1}\sup_t|F_p(t\mid x)-F_0(t\mid x)|.
\end{equation}
For a target covariate distribution $G_*$ with $dG_*/dG_0\le C_*$,
\begin{equation}\label{eq:overlap}
 d_{G_*}^2(p,q)\le C_*d_{G_0}^2(p,q).
\end{equation}
\end{proposition}
In particular, the losses
\begin{align*}
 L_F(p,p_0)&=\int\sup_t|F_p(t\mid x)-F_0(t\mid x)|^2\,G_0(dx),\\
 L_a(p,p_0)&=\int\left|\int a(y,x)(p-p_0)(y\mid x)\,\mu_y(dy)\right|^2G_0(dx)
\end{align*}
satisfy $L_F\le d_{G_0}^2$ and $L_a\le4\|a\|_\infty^2d_{G_0}^2$. Corollary~\ref{cor:risk} therefore bounds the expected posterior-average and predictive losses by $C\eps_n(\alpha_0)^2$. Under the quantile assumptions, \eqref{eq:quantile} also gives
\[
 \int\sup_{0<\tau<1}|Q_p(\tau\mid x)-Q_0(\tau\mid x)|^2\,G_0(dx)
 \le m^{-2}L_F(p,p_0).
\]
Its expected posterior-average and predictive losses have the same upper rate. Here the predictive quantile is $Q_{\bar p_n}$, the quantile of the posterior predictive distribution; it is not the posterior average of the quantiles of individual draws. Appendix~\ref{app:functional-details} proves the inequalities and verifies measurability.

These bounds average over covariates. They do not by themselves give pointwise or uniform-in-$x$ rates. The quantile condition concerns the Lebesgue density $f_0$, rather than the density $p_0$ relative to an arbitrary reference measure.

Equation~\eqref{eq:overlap} permits the same losses to be averaged under a different covariate distribution $G_*$, provided $dG_*/dG_0$ is bounded and the conditional outcome distribution is unchanged. The results concern estimation error; unbounded moments, root-$n$ inference, and coverage of credible sets require further arguments.

\subsection{Coefficient recovery and uniform convergence}
Density contraction does not immediately imply coefficient contraction. The map $h\mapsto p_h$ squares the representation, and the prior permits $h_\theta$ to take negative values. The inverse inequality in Proposition~\ref{prop:geometry}, by contrast, applies to bounded nonnegative representations.

The next assumptions allow us to bridge this gap. Within the prior class, representations that take a nonpositive value are separated in Hellinger distance from the positive truth class. Density contraction then excludes those representations, after which the inverse inequality gives coefficient recovery.

\begin{assumption}[Supports and continuity]\label{ass:compact}
The spaces $\mathcal Y$ and $\mathcal X$ are compact metric spaces, $\mathcal Y$ is connected, both reference measures have full support, and every basis function is continuous.
\end{assumption}
Continuity and full support make the basis-growth bound hold at every $(y,x)$ and the zero outcome integrals hold at every $x$. Together with $\alpha_->\rho/2$, they also give uniform convergence of the series. Connectedness of the outcome support prevents a continuous representation from changing sign in $y$ without vanishing. The integral-one restriction then fixes the sign of any representation that never vanishes.

\begin{theorem}[Coefficient and uniform rates]\label{thm:uniform}
Maintain Theorem~\ref{thm:adaptive}, impose Assumption~\ref{ass:compact}, and suppose $g_0\ge\underline g>0$. Set
\[
 u_n(\alpha_0)=\eps_n(\alpha_0)^{1-\rho/(2\alpha_0)}
             =n^{-(\alpha_0-\rho/2)/(2\alpha_0+1)}.
\]
There are uniform constants $M,c,C>0$ such that, for all sufficiently large $n$,
\begin{equation}\label{eq:uniform-contraction}
 \begin{split}
 \E_0\Pi_n\bigl\{&\|\theta-\theta_0\|_2>M\eps_n(\alpha_0)
   \ \text{or}\ \|h_\theta-h_0\|_\infty>Mu_n(\alpha_0)\\
 &\text{or}\ \|p_\theta-p_0\|_\infty>Mu_n(\alpha_0)
   \ \text{or}\ \min_z h_\theta(z)\le b_0/2\mid\cD_n\bigr\}
 \le Ce^{-ct_n(\alpha_0)}.
 \end{split}
\end{equation}
The expected posterior-average squared coefficient loss is $O(\eps_n(\alpha_0)^2)$, and the corresponding squared uniform losses are $O(u_n(\alpha_0)^2)$. The constants and the threshold for $n$ may depend on the fixed supports, reference measures, and basis in Assumption~\ref{ass:compact}, as well as $\underline g$ and the fixed quantities in Theorem~\ref{thm:adaptive}. They are uniform over that theorem's smoothness and truth classes and $G_0\in\Gclass(\underline g,\overline g)$.
\end{theorem}
The coefficient rate is the same as the integrated-Hellinger rate. The uniform rate is slower: it controls the largest error over all outcomes and covariates, and depends on the basis-growth exponent $\rho$. The condition $\alpha_0>\rho/2$ makes $1-\rho/(2\alpha_0)$ positive, so $u_n(\alpha_0)\to0$.

Appendix~\ref{app:functional-details} first excludes nonpositive representations, then applies the inverse Hellinger inequality and interpolation. The complexity bound in Theorem~\ref{thm:adaptive} ensures that retained smoothness indices are no more than $O(1/\log n)$ below $\alpha_0$, which permits the stated uniform rate under the adaptive prior. These are upper bounds; minimax optimality under uniform loss is not established here.

\section{Polynomial bases and smoothness}\label{sec:polynomials}
\subsection{Polynomial degree and contraction rates}
The smoothness index $\alpha$ used so far describes decay with respect to the position of a coefficient in the chosen ordering. To interpret it for polynomial series, we now relate coefficient position to polynomial degree. The number of terms of a given maximum degree depends on the total number of outcome and covariate coordinates.

Let $\mathcal Y=[-1,1]$, $\mathcal X=[-1,1]^{d_x}$, and $q=d_x+1$. For a tensor polynomial $\Phi_\nu$ with multi-index $\nu\in\N_0^q$, retain $\nu_1\ge1$ so that its outcome integral is zero. Order the terms by total degree $|\nu|_1$, with a fixed rule for ties. The number of terms through degree $M$ is
\begin{equation}\label{eq:count}
 J_M=\binom{M+q}{q}-\binom{M+q-1}{q-1}
     =\binom{M+q-1}{q}\asymp M^q.
\end{equation}
Thus retaining terms through degree $M$ uses order $M^q$ coefficients. Define
\begin{equation}\label{eq:degree-norm}
 \|h-1\|_{s,\mathrm{deg}}^2
 =\sum_{\nu:\nu_1\ge1}(1+|\nu|_1)^{2s}\theta_\nu^2.
\end{equation}
The index $s$ controls decay with polynomial degree, rather than with coefficient position. If $d(j)$ is the degree of term $j$, \eqref{eq:count} gives $1+d(j)\asymp j^{1/q}$. Consequently,
\[
 \|h-1\|_{s,\mathrm{deg}}\asymp\|\theta\|_{2,s/q},
 \qquad \alpha=s/q,
\]
with comparison constants uniform on compact intervals of $s$. In particular, $j^{1/q}\le1+d(j)$ gives $\|\theta\|_{2,s/q}\le\|h-1\|_{s,\mathrm{deg}}$. A degree-$M$ truncation therefore has approximation error of order $M^{-s}$ and uses order $M^q$ terms. Balancing $M^{-s}$ with $\sqrt{M^q/n}$ gives the rate $n^{-s/(2s+q)}$.

An orthogonal change of basis within a degree level preserves the degree norm. The index-weighted norm and the prior can nevertheless depend on the ordering within that level. The remaining basis-specific requirement is the growth bound in Assumption~\ref{ass:kernel}, which the next proposition verifies.

\begin{proposition}[Legendre and Chebyshev basis bounds]\label{prop:classical}
For tensor Legendre polynomials under uniform product probability measure, total-degree ordering satisfies Assumption~\ref{ass:kernel} with $\rho=2$. At $z_*=(1,\ldots,1)$,
\begin{equation}\label{eq:legendre-sharp}
 cJ^2\le\sum_{j\le J}\phi_j(z_*)^2
 \le\sup_z\sum_{j\le J}\phi_j(z)^2\le CJ^2.
\end{equation}
For tensor Chebyshev polynomials of the first kind under product arcsine probability measure, any ordering of distinct terms satisfies
\begin{equation}\label{eq:cheb-sharp}
 J\le\sum_{j\le J}\phi_j(z_*)^2
 \le\sup_z\sum_{j\le J}\phi_j(z)^2\le2^qJ.
\end{equation}
Hence $\rho=1$ for Chebyshev series. Both growth exponents are sharp.
\end{proposition}

\begin{corollary}[Rates for polynomial series]\label{cor:physical}
Consider the degree classes \eqref{eq:degree-norm} with fixed positive radius and $h\ge b_0$ for fixed $0<b_0<1$. Choose the prior radius so that the corresponding true coefficient class lies strictly inside the prior ellipsoid, using the norm comparison above. If $g_0\le\overline g$, the prior of Section~\ref{sec:adaptive} attains the adaptive integrated-Hellinger rate
\begin{equation}\label{eq:physical-rate}
 n^{-s/(2s+q)}=n^{-s/(2s+d_x+1)},
\end{equation}
and posterior-average and predictive squared Hellinger risks are bounded by $Cn^{-2s/(2s+q)}$. If also $g_0\ge\underline g>0$, the minimax squared risk is of order $n^{-2s/(2s+q)}$. The bounds are uniform over the corresponding covariate classes and compact smoothness intervals with $s_->q$ for Legendre series and $s_->q/2$ for Chebyshev series.
\end{corollary}

The restrictions on $s$ in Corollary~\ref{cor:physical} ensure that the full coefficient ellipsoid consists of uniformly bounded representations. The next result shows that the corresponding thresholds cannot be lowered for this boundedness property.

\begin{proposition}[Thresholds for uniform boundedness]\label{prop:embedding-sharp}
For $\alpha\ge0$, the bound
\[
 \sup_{N\ge1}\ \sup_{\sum_{j\le N}j^{2\alpha}\theta_j^2\le1}
 \left\|\sum_{j\le N}\theta_j\phi_j\right\|_\infty<\infty
\]
holds if and only if $\alpha>1$ for the Legendre ordering, and if and only if $\alpha>1/2$ for Chebyshev. The same thresholds hold for the essential supremum. Above them, the infinite series converges uniformly and defines a bounded linear map from the weighted coefficient space to the space of bounded functions.
\end{proposition}
Appendix~\ref{app:polynomial-details} proves the basis bounds and the boundedness thresholds. Failure of a uniform bound over the entire ellipsoid does not mean that each of its members is unbounded. Section~\ref{sec:floor} imposes boundedness directly on the truth class and obtains the integrated-Hellinger minimax rate at every positive degree smoothness.

Legendre and Chebyshev expansions use different reference measures and define different smoothness classes. Their different boundedness thresholds therefore do not by themselves rank the two procedures on a common class of Lebesgue densities.

\subsection{A differentiability condition for Legendre series}
The preceding results impose a restriction on polynomial coefficients. For Legendre series, a sufficient condition can also be stated in terms of derivatives of the true density: positivity and $2m$ continuous derivatives imply finite degree norm of order $2m$ for its normalized square root. To establish this connection, we use a differential operator for which Legendre polynomials are eigenfunctions. Define
\[
 \mathcal L v=-\sum_{r=1}^q\partial_{z_r}\{(1-z_r^2)\partial_{z_r}v\}.
\]
The tensor Legendre polynomials satisfy $\mathcal L\Phi_\nu=\lambda_\nu\Phi_\nu$, where $\lambda_\nu=\sum_r\nu_r(\nu_r+1)$. For $v\in L_2(\mu)$, define $(I+\mathcal L)^{s/2}v$ as the $L_2(\mu)$ sum of the series with coefficients $(1+\lambda_\nu)^{s/2}\langle v,\Phi_\nu\rangle$, whenever that series converges. Denote its domain by $D((I+\mathcal L)^{s/2})$. For general $v\in L_2(\mu)$, define
\[
 \|v\|_{s,\mathrm{deg}}^2=\sum_{\nu\in\N_0^q}(1+|\nu|_1)^{2s}|\langle v,\Phi_\nu\rangle|^2.
\]
For $v=h-1$, the terms with $\nu_1=0$ vanish.

\begin{proposition}[Degree smoothness and derivatives]\label{prop:sobolev}
For $s\ge0$, finite degree norm is equivalent to membership in $D((I+\mathcal L)^{s/2})$, and
\begin{equation}\label{eq:spectral-norm}
 \|(I+\mathcal L)^{s/2}v\|_2^2
 \le\|v\|_{s,\mathrm{deg}}^2
 \le(q+1)^s\|(I+\mathcal L)^{s/2}v\|_2^2.
\end{equation}
For polynomials,
\begin{equation}\label{eq:weighted-energy}
 \|(I+\mathcal L)^{1/2}v\|_2^2
 =\|v\|_2^2+\sum_{r=1}^q\int(1-z_r^2)|\partial_{z_r}v|^2\,d\mu.
\end{equation}
If a conditional density relative to uniform $\mu_y$ is positive on the cube and belongs to $C^{2m}([-1,1]^q)$, then its normalized square root has finite degree norm of order $2m$, for any integer $m\ge1$.
\end{proposition}
Here $C^{2m}$ means that the function extends to a $C^{2m}$ function on a neighborhood of the cube. Positivity and compactness give a positive lower bound, which permits differentiation of the normalized square root. Appendix~\ref{app:polynomial-details} proves the result and extends \eqref{eq:weighted-energy} to the completion of polynomials under its displayed norm. The identity weights each squared derivative by $1-z_r^2$. It therefore describes a weighted derivative norm; no equivalence with ordinary unweighted Sobolev or H\"older balls at every fractional order is claimed.

\subsection{Known relative smoothness across coordinates}
Fix $a_1,\ldots,a_q>0$, put $D_a=\sum_r a_r^{-1}$, and order terms by
\[
 \Lambda_\nu=\max_{1\le r\le q}(1+\nu_r)^{a_r}.
\]
The known weights $a_r$ specify relative smoothness across coordinates. For example, $a_2=2a_1$ means that the smoothness indices $s_r=ua_r$ assign twice as much smoothness to the second coordinate as to the first. Only the common scale $u$ is unknown.

The number of retained terms with $\Lambda_\nu\le T$ is of order $T^{D_a}$. Thus $D_a$ plays the counting role that $q$ played for total-degree ordering, and the corresponding coefficient index is $\alpha=u/D_a$.

\begin{corollary}[Adaptation to a common smoothness scale]\label{cor:anisotropic}
Suppose $\sum_{\nu:\nu_1\ge1}\Lambda_\nu^{2u}\theta_\nu^2\le R^2$ and $h\ge b_0$, with fixed $R>0$ and $0<b_0<1$. If $g_0\le\overline g$, choose the prior radius using the comparison with the index ellipsoid so that the true class lies strictly inside it. Theorem~\ref{thm:adaptive} then applies with $\alpha=u/D_a$. Writing $s_r=ua_r$, the adaptive rate is
\begin{equation}\label{eq:anisotropic-rate}
 n^{-u/(2u+D_a)}=n^{-1/(2+\sum_{r=1}^q1/s_r)}.
\end{equation}
Posterior-average and predictive squared Hellinger risks are bounded by a constant times the square of this rate. If also $g_0\ge\underline g>0$, the minimax squared risk has this order. For fixed $(a_1,\ldots,a_q)$, the conclusions hold uniformly over the corresponding covariate classes and compact intervals with $u_->D_a$ for Legendre series and $u_->D_a/2$ for Chebyshev series.
\end{corollary}
Appendix~\ref{app:polynomial-details} proves the norm and basis-growth comparisons for this ordering. The result adapts to the scalar $u$ with $(a_1,\ldots,a_q)$ fixed. Adaptation to an unrestricted smoothness vector would also require selection of the ordering.

\section{Conditional densities that may vanish}\label{sec:floor}
This section relaxes two restrictions of the main adaptive result. The true representation may vanish, and its smoothness index may lie below $\rho/2$. We instead impose an upper bound on the representation directly. A modified prior gives adaptation on this larger class up to a logarithmic factor, while preserving the exact rate on the positive, sufficiently smooth subclass.

\subsection{The truth class and regularization}
Fix
\[
 0<\alpha_-<\alpha_+<\infty,\qquad 0<R_0<R_*<\infty,\qquad B_0>1,
\]
and define
\[
 \cF_\alpha^0(R_0,B_0)=
 \{p_\theta:\theta\in\Theta_\alpha(R_0),\ 0\le h_\theta\le B_0\}.
\]
The series converges in $L_2(\mu)$, and
\[
 \int h_\theta(y,x)\,\mu_y(dy)=1
 \qquad\text{for }\mu_x\text{-almost every }x.
\]
This normalization follows from $L_2$ convergence and does not require uniform convergence. The restriction $h_\theta\le B_0$ is imposed on the truth class; it is not deduced from a lower bound on smoothness. Coefficient approximation, covering, and prior concentration remain available on every compact smoothness interval in $(0,\infty)$ without Assumption~\ref{ass:kernel}.

The difficulty is likelihood control, rather than Hellinger approximation. The local likelihood argument used previously keeps both representations away from zero. To replace that argument, we add a small reference-density component to each candidate density. For $\delta_n=n^{-3}$, introduce two specifications:
\begin{equation}\label{eq:two-links}
 p_{0,\theta,n}=p_\theta,\qquad
 p_{1,\theta,n}=(1-\delta_n)p_\theta+\delta_n.
\end{equation}
Both specifications integrate to one with respect to $\mu_y$. The regularized density $p_{1,\theta,n}$ is at least $\delta_n$ because the constant function one is the density of the reference probability measure. If $\mu_y(dy)=w_y(y)dy$, its added component in the Lebesgue density is $\delta_n w_y(y)$.

The regularized specification changes a candidate density by at most $2\delta_n$ in squared Hellinger distance. Moreover, $p_0\le B_0^2$ and $p_{1,\theta,n}\ge\delta_n$ imply
\[
 \log\{p_0/p_{1,\theta,n}\}
 \le\log(B_0^2/\delta_n)
 \qquad P_0\text{-almost surely}.
\]
This one-sided bound permits the likelihood argument to accommodate zeros in the true density. The approximation bias and likelihood bounds are proved in Appendix~\ref{app:floor}.

\subsection{The prior and posterior rates}
The prior assigns probability to both specifications. The regularized one supplies approximations when the truth may vanish; retaining the original one permits the sharper bound on the positive subclass. Use the grid \eqref{eq:grid} on this wider smoothness interval and set
\begin{align}
 t_{0,n}(\alpha)&=n^{1/(2\alpha+1)},&
 t_{1,n}(\alpha)&=\{n(\log n)^2\}^{1/(2\alpha+1)},\label{eq:two-complexities}\\
 w_n(\ell,\alpha)&=\frac{e^{-At_{\ell,n}(\alpha)}}
 {\sum_{r=0}^1\sum_{\beta\in\cA_n}e^{-At_{r,n}(\beta)}},
 &\ell&\in\{0,1\}.\label{eq:joint-link-weights}
\end{align}
The larger benchmark $t_{1,n}(\alpha)$ accounts for the logarithmic factor in the likelihood bound. For a regularized approximation using $J$ terms, that bound has order $(\log n)^2\{J^{-2\alpha}+\delta_n\}$. Balancing its leading term with $J/n$ gives
\[
 (\log n)^2J^{-2\alpha}\asymp J/n,
 \qquad J\asymp t_{1,n}(\alpha).
\]
The term involving $\delta_n=n^{-3}$ is negligible at these rates. As before, these are benchmark dimensions, not fixed truncation levels under the prior.

Conditional on $(\ell,\alpha)$, draw $J$ and $\theta$ from \eqref{eq:Jprior}--\eqref{eq:coeffprior}. Denote the joint prior on $(\ell,\alpha,J,\theta)$ and its induced density prior by $\Pi_n^{\mathrm{mix}}$. Its posterior predictive density is the posterior mean of $p_{\ell,\theta,n}$.

\begin{theorem}[Adaptation when the truth may vanish]\label{thm:floor}
Let the orthonormal system satisfy the measurability and zero outcome-integral conditions of Section~\ref{sec:model}, and suppose $g_0\le\overline g$. For every fixed $A>0$, there are uniform constants $M,c,C>0$ such that, for all sufficiently large $n$,
\begin{equation}\label{eq:floor-bound}
 \E_0\Pi_n^{\mathrm{mix}}\{d_{G_0}(p_{\ell,\theta,n},p_0)>M\eps_n^+(\alpha_0)\mid\cD_n\}
 \le Ce^{-ct_{1,n}(\alpha_0)/\log n},
\end{equation}
where
\begin{equation}\label{eq:floor-rate}
 \eps_n^+(\alpha)=\sqrt{t_{1,n}(\alpha)/n}
 =n^{-\alpha/(2\alpha+1)}(\log n)^{1/(2\alpha+1)}.
\end{equation}
The bound is uniform over $\alpha_0\in[\alpha_-,\alpha_+]$, $p_0\in\cF_{\alpha_0}^0(R_0,B_0)$, and covariate densities bounded above by $\overline g$. Expected posterior-average squared Hellinger loss and predictive squared Hellinger risk are $O(\eps_n^+(\alpha_0)^2)$. The constants and the threshold for $n$ may depend on the fixed smoothness endpoints, $R_0,R_*,B_0,\overline g$, and the prior parameters. This assertion requires neither Assumption~\ref{ass:kernel} nor completeness of the system. The same bounds hold if the true class is enlarged to allow signed representations with $\theta_0\in\Theta_{\alpha_0}(R_0)$ and $|h_{\theta_0}|\le B_0$.

Suppose in addition that Assumption~\ref{ass:kernel} holds. Fix $\alpha_p\in[\alpha_-,\alpha_+]$ with $\alpha_p>\rho/2$ and $b_0>0$. On the subclass $\alpha_0\in[\alpha_p,\alpha_+]$ and $h_0\ge b_0$, the same prior satisfies
\begin{equation}\label{eq:mixed-exact-bound}
 \E_0\Pi_n^{\mathrm{mix}}\{d_{G_0}(p_{\ell,\theta,n},p_0)>M\eps_n(\alpha_0)\mid\cD_n\}
 \le Ce^{-ct_{0,n}(\alpha_0)},
\end{equation}
and both squared-risk bounds are $O(\eps_n(\alpha_0)^2)$. The constants and the threshold for $n$ in this assertion may also depend on $b_0$, $\alpha_p$, and the basis-growth constants $(C_K,\rho)$. The prior does not use $b_0$ or $\alpha_p$.
\end{theorem}
The first assertion permits zeros and every positive coefficient smoothness, but its rate contains a logarithmic factor. The second assertion uses the same prior and removes that factor on the positive, sufficiently smooth subclass. Neither $b_0$ nor the lower smoothness threshold $\alpha_p$ for that subclass enters the prior.

Nonnegativity makes $h_0$ the unique normalized square root. The first assertion itself needs only a bounded representation with outcome integral one, which explains the extension to signed representations in the theorem. Appendix~\ref{app:floor} proves both assertions.

The first assertion also holds using only the regularized specification, with smoothness probabilities
\[
 w_n^{\mathrm{floor}}(\alpha)=
 \frac{e^{-At_{1,n}(\alpha)}}{\sum_{\beta\in\cA_n}e^{-At_{1,n}(\beta)}}.
\]
The conclusion without a logarithmic factor uses the joint prior, which assigns positive probability to the original specification. The theorem concerns density estimation; it does not establish consistent selection of $\ell$ or a lower bound requiring the logarithmic factor.

\subsection{The minimax benchmark at low smoothness}
The preceding Bayesian upper bound does not show that its logarithmic factor is unavoidable. For Legendre and Chebyshev degree classes, the next theorem establishes the minimax benchmark at every $s>0$. Its upper bound uses a separate estimator constructed from tests; that estimator may depend on $s$. Thus this result does not give an exact adaptive Bayesian rate on the full bounded class.

\begin{theorem}[Minimax rates at every positive degree smoothness]\label{thm:all-s-minimax}
For the total-degree Legendre or Chebyshev system, let
\[
 \cF_{s,\mathrm{deg}}^0(R,B)=
 \left\{p_h:\|h-1\|_{s,\mathrm{deg}}\le R,\quad
 0\le h\le B,\quad\int h(y,x)\,\mu_y(dy)=1\right\},
 \qquad R>0,\ B>1.
\]
For fixed $q,R,B,\underline g,\overline g$, uniformly over $s$ in any compact interval $[s_-,s_+]\subset(0,\infty)$ and $G_0\in\Gclass(\underline g,\overline g)$,
\begin{equation}\label{eq:all-s-minimax}
 \inf_{\widehat p_n}\sup_{p_0\in\cF_{s,\mathrm{deg}}^0(R,B)}
 \E_{p_0,G_0}^n d_{G_0}^2(\widehat p_n,p_0)
 \asymp n^{-2s/(2s+q)}.
\end{equation}
There is an estimator attaining the upper bound that does not use $G_0$, $\underline g$, or $\overline g$, although it may depend on $s$ and the fixed radius $R$. The risk constants may depend on the two covariate-density bounds. The same upper risk bound holds over the whole degree ellipsoid of $L_2(\mu)$ representations with outcome integral one, including signed or unbounded representations.

If the smoothness interval of the prior in Theorem~\ref{thm:floor} contains $[s_-/q,s_+/q]$ and $R_*>R$, that prior adapts over the bounded nonnegative class at rate
\begin{equation}\label{eq:all-s-bayes-rate}
 n^{-s/(2s+q)}(\log n)^{q/(2s+q)}.
\end{equation}
For each fixed $b_0>0$, it attains the rate without the logarithmic factor on the subclass $h\ge b_0$ when $s>q$ for Legendre and $s>q/2$ for Chebyshev. This last conclusion is uniform on compact intervals separated from the respective thresholds.
\end{theorem}
Appendix~\ref{app:all-s} constructs the estimator using tests among finitely many approximating densities. For the lower bound, it constructs localized perturbations whose integrals over outcomes equal zero. Their supports are disjoint, so their sum remains uniformly bounded at every $s>0$. This gives the lower bound below the thresholds for uniform boundedness of the full ellipsoid. The resulting minimax benchmark applies at every positive smoothness; the Bayesian bound on the general bounded class is established up to the displayed logarithmic factor.

\section{Independent Gaussian coefficients}\label{sec:ordinary}
The exact-rate prior in Section~\ref{sec:adaptive} uses smoothness-dependent Gaussian variances, conditions the coefficients on an ellipsoid, and mixes over a sample-size-dependent smoothness grid. We now consider a simpler specification: independent Gaussian coefficients with a common variance and a prior only on the truncation level. This prior does not use a smoothness index or an ellipsoid restriction and is independent of sample size. The bounds below retain logarithmic factors.

For fixed $\lambda,\tau>0$, let
\begin{equation}\label{eq:ordinaryprior}
 \pi_\zeta(J)=Z_\zeta^{-1}e^{-\lambda J[\log(e+J)]^\zeta},\qquad
 \theta_j\mid J\stackrel{\mathrm{iid}}\sim N(0,\tau^2),\quad j\le J,
 \qquad 0\le\zeta\le1.
\end{equation}
Here $J\in\N_0$, $Z_\zeta$ is the finite normalizing sum over $J$, and coefficients above $J$ are zero. The conditional density is $p_\theta$ in \eqref{eq:series-model}. The fixed parameter $\zeta$ controls how strongly the prior penalizes large truncation levels: $\zeta=0$ gives a geometric tail, whereas $\zeta=1$ gives a penalty of order $J\log J$ in the negative log prior probability.

\begin{theorem}[Contraction under independent Gaussian coefficients]\label{thm:ordinary}
Suppose $g_0\le\overline g$, Assumption~\ref{ass:kernel} holds, and $p_0\in\cF_{\alpha_0}(R_0,b_0)$ with $\rho/2<\alpha_-\le\alpha_0\le\alpha_+$. The truth radius $R_0$ may be any fixed finite nonnegative number. Define
\begin{equation}\label{eq:ordinaryrate}
 \widetilde\eps_n(\alpha)=\left(\frac n{\log n}\right)^{-\alpha/(2\alpha+1)},\qquad
 \eps_{n,\zeta}(\alpha)=\widetilde\eps_n(\alpha)(\log n)^{(1-\zeta)/2}.
\end{equation}
There are uniform constants $M,c,C>0$ such that the posterior under \eqref{eq:ordinaryprior} satisfies, for all sufficiently large $n$,
\begin{equation}\label{eq:ordinary-bound}
 \E_0\Pi\{d_{G_0}(p_\theta,p_0)>M\eps_{n,\zeta}(\alpha_0)\mid\cD_n\}
 \le Ce^{-cn\widetilde\eps_n(\alpha_0)^2}.
\end{equation}
Expected posterior-average squared Hellinger loss and predictive squared Hellinger risk are $O(\eps_{n,\zeta}(\alpha_0)^2)$. The constants are uniform over $\alpha_0\in[\alpha_-,\alpha_+]$, the stated truth classes, and covariate densities bounded above by $\overline g$, for fixed class bounds, basis-growth constants, and prior parameters. For each fixed true distribution, the posterior probability in \eqref{eq:ordinary-bound} tends to zero almost surely along an infinite iid sample.
\end{theorem}
In this proof, the negative log prior probability of a small neighborhood of the approximating coefficients is bounded at order $J\log n$, rather than order $J$ as for the smoothness-matched prior. Balancing $J^{-2\alpha}$ with $J\log n/n$ gives $\widetilde\eps_n(\alpha)$. Controlling the larger range of truncation levels allowed by a heavier dimension-prior tail gives the additional factor $(\log n)^{(1-\zeta)/2}$.

Consequently, $\zeta=1$ gives the upper rate $(n/\log n)^{-\alpha/(2\alpha+1)}$, whereas $\zeta=0$ gives this rate multiplied by $\sqrt{\log n}$. Appendix~\ref{app:ordinary} proves these bounds. They are upper bounds for this prior and proof; no matching lower bound establishes the necessity of the logarithmic factors.

\section{Selection of polynomial terms}\label{sec:sparse}
\subsection{A prior on subsets}
The preceding priors retain the first $J$ terms in a fixed ordering. A representation may instead be well approximated by a small number of terms scattered through that ordering. We therefore place a prior on subsets of a finite dictionary. Sparsity here refers to polynomial terms, not directly to covariates: selecting a term need not identify a unique set of relevant covariates.

Let $\mathcal B_{K_n}=\{\phi_1,\ldots,\phi_{K_n}\}$, where $1\le K_n\le n^\kappa$ for fixed $\kappa\ge0$. The orthonormal dictionary has zero outcome integrals and satisfies Assumption~\ref{ass:kernel}; completeness is not required. For $S\subseteq\{1,\ldots,K_n\}$ and $k=|S|$, set
\[
 c_n(S)=k\log n+\log\binom{K_n}{k}.
\]
The first term accounts for concentrating the coefficient prior near a $k$-term approximation. The second is the logarithm of the number of possible subsets of size $k$, and penalizes the choice of term locations.

For fixed $A_s>1$, draw the subset and coefficients according to
\begin{equation}\label{eq:supportprior}
 \pi_n(S)=\frac{e^{-A_sc_n(S)}}{\sum_Te^{-A_sc_n(T)}},\qquad
 \theta_j\mid S\stackrel{\mathrm{iid}}\sim N(0,\tau^2),\quad j\in S.
\end{equation}
All other coefficients are zero. In this section, $\Pi_n$ denotes the prior on $(S,\theta)$ and its induced prior on conditional densities. The normalizing sum in \eqref{eq:supportprior} runs over individual subsets, including the empty set. It lies in $[1,2]$ for $n\ge2$, as shown in Appendix~\ref{app:sparse}.

\subsection{Approximation and oracle contraction}
We compare the posterior with the best available approximation after accounting for the complexity of its selected terms. This is a comparison used to state the risk bound, not an optimization requiring knowledge of the truth to construct the prior.

Suppose the true density has normalized square-root representation $h_0$ with
\begin{equation}\label{eq:sparsetruth}
 b_0\le h_0\le B_0,\qquad \|h_0-1\|_2\le R_0,
 \qquad b_0>0.
\end{equation}
The truth need not belong to the closed span of the dictionary. Define its finite projection by
\[
 \theta_{0j}=\int(h_0-1)\phi_j\,d\mu,\qquad
 h_{0,S}=1+\sum_{j\in S}\theta_{0j}\phi_j.
\]
Orthonormality gives
\[
 \|h_{0,S}-h_0\|_2^2
 =\|h_0-1\|_2^2-\sum_{j\in S}\theta_{0j}^2.
\]
To apply the local likelihood bound, restrict the comparison to sets satisfying
\begin{equation}\label{eq:admissible}
 \cS_n^+=\{S:\|h_{0,S}-h_0\|_\infty\le b_0/4\}.
\end{equation}
For these sets, $h_{0,S}\ge3b_0/4$, so a sufficiently small coefficient neighborhood also remains positive and the local likelihood bound applies. The condition restricts the approximations against which we compare the posterior. It does not restrict prior draws, and computing the prior does not require knowing which sets are admissible. Define
\begin{equation}\label{eq:sparseoracle}
 \eta_n^2=\min_{S\in\cS_n^+}
 \left\{\|h_{0,S}-h_0\|_2^2+
 \frac{(|S|+1)\log n+\log\binom{K_n}{|S|}}{n}\right\},
\end{equation}
with $\eta_n=\infty$ if $\cS_n^+$ is empty. The minimum exists whenever $\cS_n^+$ is nonempty because the dictionary is finite. The additional $\log n/n$ term ensures $n\eta_n^2\ge\log n$, including when the minimizing set is empty.

\begin{theorem}[Oracle contraction]\label{thm:sparse}
Maintain the dictionary conditions and \eqref{eq:sparsetruth}, and suppose $g_0\le\overline g$. For every fixed $q_0>0$, there are $M,L>0$ and $n_0$ such that, whenever $n\ge n_0$ and $\eta_n\le1$,
\begin{equation}\label{eq:sparse-contraction}
 \E_0\Pi_n\{d_{G_0}(p_\theta,p_0)>M\eta_n
 \ \text{or}\ c_n(S)>Ln\eta_n^2\mid\cD_n\}
 \le4e^{-q_0n\eta_n^2}.
\end{equation}
Expected posterior-average squared Hellinger loss and predictive squared Hellinger risk are at most $(M^2+8)\eta_n^2$. Constants are uniform in $n,K_n,G_0$, and $h_0$, for fixed $q_0$, dictionary, scalar bounds, and prior parameters. In particular, the conclusions hold eventually whenever $\eta_n\to0$.
\end{theorem}
The oracle criterion combines squared approximation error, the coefficient-concentration term $(|S|+1)\log n/n$, and the term-location penalty $\log\binom{K_n}{|S|}/n$. The posterior attains an upper rate determined by their best trade-off over admissible sets. The theorem also controls posterior model complexity, but does not establish consistent selection of a particular subset. Appendix~\ref{app:sparse} gives the proof.

To interpret the bound, suppose admissible $k$-term approximations have squared error of order $k^{-2\gamma}$. Because $K_n\le n^\kappa$, the complexity contribution for nonempty sets is of order at most $k\log n/n$. Balancing these two terms gives $k\asymp(n/\log n)^{1/(2\gamma+1)}$. Here $\gamma$ describes approximation by selected terms; it need not equal the smoothness index associated with retaining the first $J$ terms.

\begin{corollary}[Rates under sparse approximation]\label{cor:bestk}
Fix $\gamma>0$, $R\ge0$, and $0<a\le b<\infty$. Put $I_n=(n/\log n)^{1/(2\gamma+1)}$. Suppose that for every sufficiently large $n$ there are an integer $k_n$ and $S_n\in\cS_n^+$ such that
\[
 aI_n\le k_n\le bI_n,\qquad |S_n|\le k_n,\qquad
 \|h_{0,S_n}-h_0\|_2^2\le R^2k_n^{-2\gamma}.
\]
Then, eventually,
\[
 \eta_n\le\{R^2a^{-2\gamma}+(\kappa+1)b+1\}^{1/2}
 \left(\frac n{\log n}\right)^{-\gamma/(2\gamma+1)},
\]
so Theorem~\ref{thm:sparse} applies.
\end{corollary}
The assumption requires one admissible approximation at each sufficiently large sample size, with $k_n$ of order $I_n$. Substitution into the oracle criterion gives the bound; see Appendix~\ref{app:selection-details}.

\subsection{A lower bound for selection}
The oracle bound includes a penalty for choosing term locations. We now show that a contribution of this order is unavoidable on a sparse class, under the stated size restriction on the dictionary. This lower bound concerns the location penalty, not the full coefficient-concentration term in the upper bound.

Let $\mathcal P_{k,K}$ be the class of normalized-square densities with at most $k$ nonzero coefficients among the first $K$ terms, all other coefficients zero, and $1/2\le h_\theta\le3/2$.

\begin{theorem}[Selection cost]\label{thm:sparse-lower}
Under Assumptions~\ref{ass:design}--\ref{ass:kernel}, there are $c_1,c_2,c_3>0$, depending only on the fixed covariate-density and basis-growth bounds, such that, for positive integers $k,K$ with $K\ge16k$, $k\log\lfloor K/k\rfloor\ge c_1$, and
\begin{equation}\label{eq:sparse-local-condition}
 kK^\rho\log(K/k)\le c_2n,
\end{equation}
we have
\begin{equation}\label{eq:sparse-lower}
 \inf_{\widehat p_n}\sup_{p_0\in\mathcal P_{k,K}}
 \E_{p_0,G_0}^n d_{G_0}^2(\widehat p_n,p_0)
 \ge c_3\frac{k\log(K/k)}{n}.
\end{equation}
\end{theorem}
Appendix~\ref{app:sparselower} constructs subsets with different term locations. Condition~\eqref{eq:sparse-local-condition} bounds the representations of the resulting alternatives between $1/2$ and $3/2$. Their Hellinger separation and the Kullback--Leibler bound relative to $p=1$ allow application of Fano's inequality. The result establishes a lower risk contribution of order $k\log(K/k)/n$ for choosing term locations. It does not establish necessity of the additional $k\log n$ term in the oracle upper bound.

\section{Conclusion}\label{sec:scope}
We study conditional densities represented by normalized squares of orthogonal series. For positive, sufficiently smooth representations, a prior on smoothness, truncation level, and ellipsoid-restricted Gaussian coefficients gives adaptive posterior contraction at the minimax rate without a logarithmic factor. Its posterior predictive density attains the matching squared Hellinger risk.

Allowing a small mixture with the reference density extends the Bayesian result to bounded representations that may vanish and to every positive coefficient smoothness. The resulting bound contains a logarithmic factor, while the same prior retains the exact rate on the positive, sufficiently smooth subclass. For bounded nonnegative Legendre and Chebyshev degree classes, a separate testing-based estimator establishes the exact minimax benchmark at every $s>0$.

\clearpage
\appendix

\section{Normalization and approximation}\label{app:additional}
\subsection{Reference measures and measurable versions}\label{app:versions}
Suppose $\mu_y(dy)=w_y(y)\,dy$. Then $f(y\mid x)=w_y(y)p(y\mid x)$ represents the same conditional distribution as $p$. Substitution in \eqref{eq:hell}--\eqref{eq:var} leaves all three quantities unchanged, including when a logarithmic integral is infinite. Under any distribution dominated by $\mu_y\otimes G_0$, $0<w_y(Y)<\infty$ almost surely. The factor $\prod_{i=1}^n w_y(Y_i)$ therefore cancels from the posterior for every finite sample. This argument also applies to a prior that depends on $n$.

For $\theta\in\ell_2$, write $h_{\theta,m}=1+\sum_{j=1}^m\theta_j\phi_j$. Orthonormality gives
\[
 \|h_{\theta,m}-h_{\theta,k}\|_2^2
 =\sum_{j=k+1}^m\theta_j^2,\qquad m>k,
\]
so the partial sums converge to a measurable $h_\theta$ in $L_2(\mu)$. For every $v\in L_2(\mu)$, Cauchy--Schwarz and Fubini give
\[
 \int\left|\int v(y,x)\,\mu_y(dy)\right|^2\mu_x(dx)
 \le\int |v|^2\,d\mu.
\]
Applying this inequality to $h_{\theta,m}-h_\theta$ shows that $\int h_\theta(y,x)\,\mu_y(dy)=1$ for $\mu_x$-almost every $x$. Fubini gives $Q_\theta(x)<\infty$ almost everywhere, and Cauchy--Schwarz gives $Q_\theta(x)\ge1$ there. On the exceptional covariate set, set $h_\theta(\cdot,x)=1$. Then $p_\theta(\cdot\mid x)$ integrates to one for every $x$. The change preserves coefficient integrals, joint distributions, and losses whenever $G\ll\mu_x$.

The basis functions can likewise be chosen so that
\[
 \int\phi_j(y,x)\,\mu_y(dy)=0,\qquad
 \int\phi_j(y,x)^2\,\mu_y(dy)<\infty
 \quad\text{for every }j\text{ and }x.
\]
Set all basis functions to zero on the union of their exceptional covariate sets. That union is $\mu_x$-null because the system is countable. Every finite expansion is then jointly measurable in its coefficients and observations and satisfies the required outcome integrals at every $x$. All priors in the paper assign probability one to finite expansions, so these versions suffice to define their likelihoods. Infinite expansions are used for approximation and for the parameter classes; their versions on null sets do not affect the risk bounds.

For the normalized square root in \eqref{eq:canonical}, Tonelli gives
\[
 \|h_p\|_2^2=\int c_p(x)^{-2}\,\mu_x(dx).
\]
Thus $h_p$ belongs to $L_2(\mu)$ if and only if the displayed integral is finite. Existence of the normalized representation does not establish this integrability condition. These normalization arguments require neither completeness of the orthonormal system nor uniform convergence.

\subsection{Normalization and identification}\label{app:model-details}
\begin{proof}[Proof of Proposition~\ref{prop:geometry}]
Fix $x$ for which the assumed outcome integrals are finite. All norms in this paragraph are in $L_2(\mu_y)$. Cauchy--Schwarz gives $Q_h\ge(\int h\,d\mu_y)^2=1$, and likewise $Q_k\ge1$. Put $a=\|h\|_2$, $b=\|k\|_2$, $r=|h|/a$, and $s=|k|/b$. Then $r=\sqrt{p_h(\cdot\mid x)}$, $s=\sqrt{p_k(\cdot\mid x)}$, and
\[
 ab\|r-s\|_2^2
 =\big\||h|-|k|\big\|_2^2-(a-b)^2
 \le\|h-k\|_2^2.
\]
Since $ab\ge1$, integration over $G_0$ and the bound $g_0\le\overline g$ give $d_{G_0}^2(p_h,p_k)\le\overline g\|h-k\|_{L_2(\mu)}^2$. This proves \eqref{eq:global-H}.

Now suppose $0\le h,k\le B$. At each of the same covariate values, $r=h/a$, $s=k/b$, $\int r\,d\mu_y=1/a$, and $\int s\,d\mu_y=1/b$. Consequently,
\[
 (a-b)^2
 =a^2b^2\left\{\int(r-s)\,d\mu_y\right\}^2
 \le a^2b^2\|r-s\|_2^2.
\]
Moreover,
\[
 a^2=\int h^2\,d\mu_y\le B\int h\,d\mu_y=B,
 \qquad b^2\le B.
\]
The identity $\|h-k\|_2^2=(a-b)^2+ab\|r-s\|_2^2$ therefore bounds the squared representation distance by $(B^2+B)\|r-s\|_2^2$. Integrating over $\mu_x$ and using $g_0\ge\underline g$ proves \eqref{eq:inverse-H}.

For an arbitrary conditional density $p$, its normalization and Cauchy--Schwarz imply $0<c_p(x)\le1$ for almost every $x$. The function $h_p=\sqrt p/c_p$ is measurable and satisfies
\[
 \int h_p(y,x)\,\mu_y(dy)=1,\qquad
 Q_{h_p}(x)=c_p(x)^{-2}<\infty,\qquad
 p_{h_p}=p.
\]
If a nonnegative $h$ has outcome integral one and generates $p$, then $h(\cdot,x)=\sqrt{Q_h(x)}\sqrt{p(\cdot\mid x)}$. Integrating over $y$ fixes $\sqrt{Q_h(x)}=c_p(x)^{-1}$, proving uniqueness. Finally, $p\le M$ implies $\sqrt p\ge p/\sqrt M$ and hence $c_p\ge M^{-1/2}$. Combining this with $c_p\le1$ and $p\ge m$ gives $\sqrt m\le h_p\le M$.
\end{proof}
\begin{proof}[Proof of Proposition~\ref{prop:identification}]
By Fubini, for almost every $x$,
\[
 \frac{h(y,x)^2}{Q_h(x)}-\frac{k(y,x)^2}{Q_k(x)}=0
 \qquad\text{for }\mu_y\text{-almost every }y.
\]
The left side is a polynomial in $y$. A nonzero univariate polynomial has finitely many roots. If it vanished almost everywhere, $\mu_y$ would be concentrated on that finite set, contrary to its infinite support. The polynomial is therefore identically zero. Factoring gives, at each such $x$,
\[
 \frac{h(\cdot,x)}{\sqrt{Q_h(x)}}
 =\frac{k(\cdot,x)}{\sqrt{Q_k(x)}}
 \quad\text{or}\quad
 \frac{h(\cdot,x)}{\sqrt{Q_h(x)}}
 =-\frac{k(\cdot,x)}{\sqrt{Q_k(x)}}.
\]
The negative identity is impossible because both outcome integrals are positive. Integrating the positive identity gives $Q_h(x)=Q_k(x)$, and hence $h(\cdot,x)=k(\cdot,x)$. Equality of two finite expansions in the same orthonormal basis then implies equality of their coefficients after padding with zeros.

For the tensor expansion, collect coefficients with the same outcome degree to obtain $h_\theta=1+\sum_{k\ge1}a_k(x)\psi_k(y)$. The sum is finite. Orthogonality in $L_2(\mu_y)$ removes its cross terms and gives $Q_\theta(x)=1+\sum_{k\ge1}a_k(x)^2$, as in \eqref{eq:normalizer}.
\end{proof}

\subsection{Approximation, covering numbers, and likelihood bounds}\label{app:approximation}
\begin{lemma}[Approximation and interpolation]\label{lem:interpolation}
Under Assumption~\ref{ass:kernel}, there is a constant $C$, depending only on $C_K,\rho,\alpha_-$, such that, uniformly for $\alpha\ge\alpha_->\rho/2$,
\begin{align}
 \esssup_z\sum_{j>m}j^{-2\alpha}\phi_j(z)^2&\le Cm^{\rho-2\alpha},\qquad m\ge1,\label{eq:weightedtail}\\
 \|\theta-\theta^{(J)}\|_2&\le R(J+1)^{-\alpha},\qquad\theta\in\Theta_\alpha(R),\label{eq:tail}\\
 \left\|\sum_{j\ge1}\delta_j\phi_j\right\|_\infty
 &\le C\|\delta\|_2^{1-\rho/(2\alpha)}\|\delta\|_{2,\alpha}^{\rho/(2\alpha)},
 \qquad\|\delta\|_{2,\alpha}<\infty.\label{eq:interpolation}
\end{align}
Consequently, $\|h_\theta\|_\infty\le1+CR$ on $\Theta_\alpha(R)$. If $g_0\le\overline g$, then
\begin{equation}\label{eq:approxH}
 d_{G_0}(p_{\theta^{(J)}},p_\theta)
 \le\sqrt{\overline g}\,R(J+1)^{-\alpha}.
\end{equation}
On compact supports, if the basis functions are continuous and \eqref{eq:kernel} holds pointwise, then $\{h_\theta:\theta\in\Theta_\alpha(R)\}$ is compact in the uniform norm.
\end{lemma}
\begin{proof}
Partition $j>m$ into the blocks $2^\ell m<j\le2^{\ell+1}m$, $\ell\ge0$. By \eqref{eq:kernel}, the contribution of block $\ell$ to the sum in \eqref{eq:weightedtail} is at most
\[
 C_K2^\rho m^{\rho-2\alpha}2^{\ell(\rho-2\alpha)}.
\]
The geometric series is bounded uniformly over $\alpha\ge\alpha_-$ because $\rho-2\alpha_-<0$. This proves \eqref{eq:weightedtail}. For the omitted coefficients, $j^{2\alpha}\ge(J+1)^{2\alpha}$ gives
\[
 \sum_{j>J}\theta_j^2
 \le(J+1)^{-2\alpha}\sum_{j>J}j^{2\alpha}\theta_j^2
 \le R^2(J+1)^{-2\alpha}.
\]
This proves \eqref{eq:tail}, which in fact holds for every $\alpha>0$.

To prove interpolation, put $u=\|\delta\|_2$ and $v=\|\delta\|_{2,\alpha}$. If $u=0$, the conclusion follows immediately. Otherwise $v\ge u>0$. Apply Cauchy--Schwarz separately to the first $m$ coefficients and the remaining coefficients, using \eqref{eq:kernel} and \eqref{eq:weightedtail}:
\[
 \left\|\sum_j\delta_j\phi_j\right\|_\infty
 \le C\{u m^{\rho/2}+v m^{\rho/2-\alpha}\}.
\]
Taking $m=\lceil(v/u)^{1/\alpha}\rceil$ proves \eqref{eq:interpolation}, with a constant independent of $\alpha\ge\alpha_-$. The same tail calculation gives
\[
 \|h_\theta-h_{\theta^{(J)}}\|_\infty
 \le CRJ^{\rho/2-\alpha}.
\]
Applying \eqref{eq:interpolation} to $\delta=\theta$ bounds $\|h_\theta\|_\infty$. Equation~\eqref{eq:approxH} follows from \eqref{eq:tail} and Proposition~\ref{prop:geometry}.

For compactness, take any sequence of coefficient vectors in $\Theta_\alpha(R)$ and extract a coordinatewise convergent subsequence. Fatou's lemma places its limit in the same ellipsoid. Every fixed finite expansion converges uniformly along the subsequence, and the displayed tail bound is uniform over the ellipsoid and tends to zero. The full expansions therefore converge uniformly to the expansion of the limiting coefficients.
\end{proof}

\begin{lemma}[Ellipsoid covering number]\label{lem:entropy}
For $\alpha$ in a fixed compact subinterval of $(0,\infty)$ and $0<\eps<R$,
\begin{equation}\label{eq:entropy}
 \log N(\eps,\Theta_\alpha(R),\|\cdot\|_2)
 \le C(R/\eps)^{1/\alpha}.
\end{equation}
The constant is uniform for $R$ in compact subintervals of $(0,\infty)$. If $g_0\le\overline g$, the same upper bound, with a constant that may also depend on $\overline g$, applies to the induced density class in $d_{G_0}$. No basis-growth assumption is required.
\end{lemma}
\begin{proof}
Set $m=\lceil(4R/\eps)^{1/\alpha}\rceil$. By \eqref{eq:tail}, truncation after $m$ coefficients incurs $\ell_2$ error at most $\eps/4$. The projected ellipsoid $E_m\subset\R^m$ has semiaxes $a_j=Rj^{-\alpha}$. Choose a maximal set of points in $E_m$ separated by more than $\eps/2$. The Euclidean balls of radius $r=\eps/4$ about these points have disjoint interiors and lie in
\[
 E_m+rB_2^m\subseteq(1+r/a_m)E_m,
\]
where $B_2^m$ is the Euclidean unit ball. If $N_m$ is the number of selected points, volume comparison gives
\[
 N_m\le(1+r/a_m)^m\prod_{j=1}^m(a_j/r).
\]
The ratio $r/a_m$ is bounded uniformly over the stated ranges. Also, $\log(m!)\ge m\log m-m$ implies
\[
 \sum_{j=1}^m\log(a_j/r)
 \le m\log\{R/(rm^\alpha)\}+\alpha m\le Cm.
\]
Hence $\log N_m\le Cm$. Maximality makes the selected points an $\eps/2$-net of $E_m$; adding the truncation error gives an $\eps$-net of the infinite ellipsoid. Since $m\lesssim(R/\eps)^{1/\alpha}$ uniformly, \eqref{eq:entropy} follows. Proposition~\ref{prop:geometry} maps a coefficient net of radius $\eps/\sqrt{\overline g}$ to a density net of radius $\eps$, proving the final assertion.
\end{proof}

\begin{lemma}[Local likelihood approximation]\label{lem:localKL}
Suppose Assumption~\ref{ass:kernel} holds, $\theta_0\in\Theta_\alpha(R_0)$, $h_{\theta_0}\ge b_0>0$, and $\theta\in\Theta_\alpha(R_*)$. There are $r_0,C,L>0$, uniform for $\alpha\ge\alpha_->\rho/2$ and $g_0\le\overline g$, such that $\|\theta-\theta_0\|_2\le r_0$ implies
\begin{equation}\label{eq:localKL}
 K_{G_0}(p_0,p_\theta)+V_{G_0}(p_0,p_\theta)
 \le C\|\theta-\theta_0\|_2^2,
 \qquad\|\log(p_0/p_\theta)\|_\infty\le L.
\end{equation}
These constants may depend on $C_K,\rho,\alpha_-,R_0,R_*,b_0,\overline g$. For any two representations with outcome integral one and values in a fixed interval $[b,B]$, $b>0$, the same bounds hold with squared $L_2(\mu)$ distance between the representations and constants depending only on $b,B,\overline g$.
\end{lemma}
\begin{proof}
Since $\|\theta-\theta_0\|_{2,\alpha_-}\le R_*+R_0$, Lemma~\ref{lem:interpolation} at $\alpha_-$ gives $\|h_\theta-h_0\|_\infty\le b_0/2$ whenever $\|\theta-\theta_0\|_2\le r_0$, for a sufficiently small uniform $r_0$. Both representations then take values in a common interval $[b_0/2,B]$. Their normalizers lie in $[1,B^2]$, so $t=p_\theta/p_0$ lies in a fixed compact subinterval of $(0,\infty)$.

On this interval, the functions
\[
 \frac{-\log t+t-1}{(\sqrt t-1)^2},
 \qquad
 \frac{(\log t)^2}{(\sqrt t-1)^2}
\]
are bounded after assigning their continuous values at $t=1$. Moreover, normalization gives $\int p_0(t-1)\,d\mu_y\,dG_0=0$. Multiplying the two inequalities by $p_0$ and integrating therefore yields
\[
 K_{G_0}(p_0,p_\theta)+V_{G_0}(p_0,p_\theta)
 \le C d_{G_0}^2(p_0,p_\theta).
\]
Proposition~\ref{prop:geometry} bounds the right side by squared coefficient error. The fixed range of $t$ also bounds the absolute log ratio. For arbitrary representations in $[b,B]$, the same ratio calculation and Proposition~\ref{prop:geometry} apply directly, without a coefficient restriction.
\end{proof}

\section{Probability inequalities and prior concentration}\label{app:prob}
Fix $p_0$ and $G_0$. Write $P_p=P_{p,G_0}$ and $P_0=P_{p_0,G_0}$, with product distributions $P_p^n$ and $P_0^n$ for the sample. Expectations $\E_p$ and $\E_0$ refer to one observation or the sample as specified by their arguments. Define the likelihood ratio and its prior integral by
\[
 R_n(p)=\prod_{i=1}^n\frac{p(Y_i\mid X_i)}{p_0(Y_i\mid X_i)},
 \qquad D_n=\int R_n(p)\,\Pi_n(dp).
\]
The denominators $p_0(Y_i\mid X_i)$ are positive and finite $P_0^n$-almost surely. Prior integrals can equivalently be taken over the parameters that determine $p$.

For one observation,
\[
 \E_0\frac{p(Y\mid X)}{p_0(Y\mid X)}
 =\int_{\{p_0>0\}}p\,d\mu_y\,dG_0\le1.
\]
Independence gives $\E_0R_n(p)\le1$, with equality if $p_0>0$ almost everywhere under $\mu_y\otimes G_0$. Tonelli then gives $\E_0D_n\le1$, so $D_n<\infty$ almost surely. Every prior used below assigns positive probability to the reference density $p=1$, whose likelihood ratio is strictly positive almost surely. Hence $D_n>0$ almost surely, and Bayes' formula defines the posterior under each truth considered here, including truths that vanish.

\Needspace{10\baselineskip}
\subsection{Bernstein's inequality}
\begin{lemma}[Bernstein bound]\label{lem:bernstein}
Let $W_i$ be finitely many independent centered variables satisfying $W_i\le b$ almost surely and $\sum_i \E W_i^2\le v$, where $b,v\ge0$. For $u>0$ and $v+bu/3>0$,
\begin{equation}\label{eq:bernstein}
 \Prob\left(\sum_iW_i\ge u\right)
 \le\exp\left\{-\frac{u^2}{2(v+bu/3)}\right\}.
\end{equation}
If $v=0$, the probability is zero. Applying the inequality to $-W_i$ gives the lower-tail bound when $W_i\ge-b$.
\end{lemma}
\begin{proof}
For $x\le a<3$ with $a\ge0$,
\[
 e^x\le1+x+\frac{x^2}{2(1-a/3)}.
\]
For $x\le0$, this follows from $e^x\le1+x+x^2/2$. For $0\le x\le a$, expand the exponential and use $k!\ge2\cdot3^{k-2}$ for $k\ge2$. Thus, for $t\ge0$ with $tb<3$, centering and $1+z\le e^z$ imply
\[
 \E e^{tW_i}
 \le\exp\left\{\frac{t^2\E W_i^2}{2(1-tb/3)}\right\}.
\]
Independence bounds the moment generating function of the sum by $\exp\{t^2v/[2(1-tb/3)]\}$. If $v>0$, exponential Markov inequality with $t=u/(v+bu/3)$ proves \eqref{eq:bernstein}. This choice satisfies $tb<3$. If $v=0$, each $W_i=0$ almost surely.
\end{proof}

\subsection{Tests for conditional densities}
\begin{lemma}[Tests against neighborhoods of conditional densities]\label{lem:tests}
There is a numerical constant $c_T>0$ with the following property. For conditional densities $p_0,p_1$ with $D=d_{G_0}(p_0,p_1)>0$, there is a test $\varphi_{p_1}$, depending only on $p_0,p_1$ and the observations, such that
\begin{equation}\label{eq:robust-test}
 P_0^n\varphi_{p_1}\le e^{-c_TnD^2},\qquad
 \sup_{p:d_{G_0}(p,p_1)\le D/32}P_p^n(1-\varphi_{p_1})
 \le e^{-c_TnD^2}.
\end{equation}
Here $P\varphi$ denotes the expectation of the test under $P$. If $\mathcal F$ can be covered by $N$ balls of radius $r/64$ in $d_{G_0}$, there is a test $\varphi$ satisfying, after a numerical reduction of $c_T$,
\begin{equation}\label{eq:global-test}
 P_0^n\varphi\le N e^{-c_Tnr^2},\qquad
 \sup_{p\in\mathcal F:d_{G_0}(p,p_0)>r}P_p^n(1-\varphi)
 \le e^{-c_Tnr^2}.
\end{equation}
\end{lemma}
\begin{proof}
For $z=(y,x)$, abbreviate $p_j(z)=p_j(y\mid x)$ and define
\[
 T(z)=\frac{\sqrt{p_1(z)}-\sqrt{p_0(z)}}
            {\sqrt{p_1(z)}+\sqrt{p_0(z)}},
 \qquad
 \varphi_{p_1}=\ind\left\{\sum_{i=1}^nT(Z_i)>0\right\}.
\]
Set $T=0$ if both densities vanish. The score satisfies $|T|\le1$ and does not use $g_0$.

To calculate its moments, write $a=\sqrt{p_0g_0}$, $b=\sqrt{p_1g_0}$, $v=b-a$, and $s=a+b$. These are functions of $(y,x)$, integrated below with respect to $\mu$. Normalization gives $\int vs\,d\mu=\int(b^2-a^2)\,d\mu=0$, and $\int v^2\,d\mu=D^2$. Since $|v|\le s$,
\[
 \E_0T=-\frac{D^2}{2}+\frac14\int\frac{v^3}{s}\,d\mu
 \le-\frac{D^2}{4},
 \qquad
 \E_1T=\frac{D^2}{2}+\frac14\int\frac{v^3}{s}\,d\mu
 \ge\frac{D^2}{4},
\]
where $v^3/s=0$ when $s=0$. Also $\E_0T^2,\E_1T^2\le D^2$.

Now let $d_{G_0}(p,p_1)\le\delta=D/32$ and put $c=\sqrt{pg_0}$. With all norms in the next calculation taken in $L_2(\mu)$, Cauchy--Schwarz gives
\[
 |\E_pT-\E_1T|
 \le\|c-b\|_2\{\|(c-b)T\|_2+2\|bT\|_2\}
 \le\delta(\delta+2D).
\]
It follows that $\E_pT\ge D^2/8$. Moreover, $\E_pT^2=\|cT\|_2^2\le(D+\delta)^2\le2D^2$. The centered score is bounded in absolute value by two. Apply Lemma~\ref{lem:bernstein} to its upper tail under $P_0^n$ and its lower tail under $P_p^n$. This proves \eqref{eq:robust-test}, for example with $c_T=1/512$.

For each covering ball that meets $\{p\in\mathcal F:d_{G_0}(p,p_0)>r\}$, choose a point $p_j$ in that intersection. There are at most $N$ chosen points. Every alternative in the displayed set lies within $r/32$ of one of them, and $d_{G_0}(p_j,p_0)>r$. Take the maximum of the corresponding tests. The union bound controls its type-I error by $Ne^{-c_Tnr^2}$; the test for the point nearest a given alternative controls its type-II error by $e^{-c_Tnr^2}$. Although the covering uses $d_{G_0}$, each pairwise score remains independent of $g_0$.
\end{proof}

\subsection{The posterior denominator}
\begin{lemma}[A lower bound for the posterior denominator]\label{lem:denominator}
Let $B$ be prior measurable, with $\Pi_n(B)=b_n>0$. Suppose that, for every $p\in B$,
\[
 K_{G_0}(p_0,p)\le C_K'\eps^2,\qquad
 V_{G_0}(p_0,p)\le C_V\eps^2,\qquad
 \log(p_0/p)\le L\quad P_0\text{-almost surely},
\]
where $C_K',C_V,L\ge0$. Then, for $u>0$ with $C_V+Lu/3>0$,
\begin{equation}\label{eq:denominator}
 P_0^n\{D_n<b_ne^{-(C_K'+u)n\eps^2}\}
 \le\exp\left\{-\frac{u^2n\eps^2}{2(C_V+Lu/3)}\right\}.
\end{equation}
If $L=0$, then $D_n\ge b_n$ almost surely. The result also applies when $B$ is a set of prior parameters satisfying these density conditions.
\end{lemma}
\begin{proof}
Let $\Pi_B$ be the prior conditional on $B$ and define
\[
 U_i=\int_B\log\{p_0(Z_i)/p(Z_i)\}\,\Pi_B(dp).
\]
The second-moment bound justifies Fubini and Jensen and yields
\[
 0\le\E_0U_i\le C_K'\eps^2,\qquad
 \E_0U_i^2\le C_V\eps^2,\qquad
 U_i-\E_0U_i\le L.
\]
The first inequality uses nonnegativity of Kullback--Leibler divergence. Jensen's inequality applied to the likelihood integral gives $D_n\ge b_n\exp(-\sum_iU_i)$. Consequently, the event in \eqref{eq:denominator} requires $\sum_i(U_i-\E_0U_i)>un\eps^2$. The variables $U_i$ are independent, and their total variance is at most $nC_V\eps^2$. Lemma~\ref{lem:bernstein} gives the asserted probability bound. If $L=0$, then $U_i\le0$ and $\E_0U_i\ge0$, so $U_i=0$ almost surely and Jensen gives $D_n\ge b_n$.
\end{proof}

The proofs below combine denominator control with a test through one inequality. Suppose $D_n\ge d_n>0$ on an event $\mathcal E$. For a prior-measurable set $B$ and a test $0\le\varphi\le1$,
\begin{equation}\label{eq:posterior-decomposition}
 \E_0\Pi_n(B\mid\cD_n)
 \le P_0^n(\mathcal E^c)+P_0^n\varphi
       +d_n^{-1}\int_B P_p^n(1-\varphi)\,\Pi_n(dp).
\end{equation}
Indeed, split the expectation over $\mathcal E$ and $\mathcal E^c$. On $\mathcal E$, bound the contribution multiplied by $\varphi$ by $\varphi$ and divide the remaining numerator by $d_n$. Tonelli and
\[
 \E_0\{(1-\varphi)R_n(p)\}\le P_p^n(1-\varphi)
\]
then give \eqref{eq:posterior-decomposition}. The inequality, rather than equality, permits $p_0$ to vanish on a set to which $p$ assigns positive probability.

\subsection{Prior probability of approximating coefficients}
\begin{lemma}[Prior probability under smoothness-matched Gaussian coefficients]\label{lem:smallball}
Fix $0<\alpha_-<\alpha_+<\infty$, $0<R_0<R_*$, and $r\in(0,R_*-R_0)$. For $\alpha\in[\alpha_-,\alpha_+]$, $J\ge1$, and $\theta_0\in\Theta_\alpha(R_0)$, let
\[
 B_{\alpha,J}=\{\theta\in\R^J:
        \|\theta-\theta_0^{(J)}\|_2\le rJ^{-\alpha}\}.
\]
There is a constant $C$, depending only on the fixed bounds, $r$, and $\tau$, such that, uniformly in $\alpha,J,\theta_0$,
\begin{equation}\label{eq:smallball}
 B_{\alpha,J}\subset\Theta_{\alpha,J}(R_*),\qquad
 \Pi_{\alpha,J}(B_{\alpha,J})\ge e^{-CJ}.
\end{equation}
\end{lemma}
\begin{proof}
For every $\theta\in B_{\alpha,J}$,
\[
 \|\theta\|_{2,\alpha}
 \le R_0+J^\alpha\|\theta-\theta_0^{(J)}\|_2
 \le R_0+r<R_*.
\]
Thus the ball lies inside the conditioning ellipsoid. To bound its unconditioned Gaussian probability, use the cube
\[
 C_{\alpha,J}=\{\theta:|\theta_j-\theta_{0j}|\le b_J,\ j\le J\},
 \qquad b_J=rJ^{-\alpha-1/2},
\]
which is contained in $B_{\alpha,J}$. Throughout this cube,
\[
 \sum_{j\le J}j^{2\alpha+1}\theta_j^2
 \le2J\|\theta_0\|_{2,\alpha}^2
       +2b_J^2J^{2\alpha+2}
 \le2J(R_0^2+r^2).
\]
Multiply the resulting lower bound on the Gaussian density by the cube's volume:
\begin{equation}\label{eq:volume-cancel}
 \begin{split}
 \widetilde\Pi_{\alpha,J}(B_{\alpha,J})
 &\ge\exp\left\{
 J\log\frac{2r}{\tau\sqrt{2\pi}}
 +(\alpha+\tfrac12)\{\log(J!)-J\log J\}
 -\frac{J(R_0^2+r^2)}{\tau^2}
 \right\}.
 \end{split}
\end{equation}
Since $\log(J!)\ge J\log J-J$, the negative logarithm of this bound is at most $CJ$, uniformly over the smoothness interval. Conditioning can only increase the probability of a set contained in the ellipsoid:
\[
 \Pi_{\alpha,J}(B_{\alpha,J})
 =\frac{\widetilde\Pi_{\alpha,J}(B_{\alpha,J})}
        {\widetilde\Pi_{\alpha,J}(\Theta_{\alpha,J}(R_*))}
 \ge\widetilde\Pi_{\alpha,J}(B_{\alpha,J})\ge e^{-CJ}.
\]
\end{proof}

\section{Proof of adaptive estimation}\label{app:main}
\begin{proof}[Proof of Theorem~\ref{thm:adaptive}]
We establish prior probability near the truth, control the posterior denominator, and construct tests on the retained smoothness classes. Fix $\alpha_0$, $p_0\in\cF_{\alpha_0}(R_0,b_0)$, and $G_0$ with $g_0\le\overline g$. All bounds below are uniform over these choices. Let
\[
 \beta=\max\{\alpha\in\cA_n:\alpha\le\alpha_0\},\qquad
 t=t_n(\beta),\quad\eps=\eps_n(\beta),\quad J_*=\lceil t\rceil.
\]
The grid spacing gives $0\le\alpha_0-\beta<c_\Delta/\log n$. Consequently,
\begin{equation}\label{eq:gridcompare}
 1\le\frac{t_n(\beta)}{t_n(\alpha_0)}
 \le C_\Delta:=\exp\left\{\frac{2c_\Delta}{(2\alpha_-+1)^2}\right\},
 \qquad
 1\le\frac{\eps_n(\beta)}{\eps_n(\alpha_0)}
 \le\sqrt{C_\Delta}.
\end{equation}

\medskip\noindent\emph{Step 1: Prior probability and the posterior denominator.}
Fix $r\in(0,R_*-R_0)$. Since $\theta_0\in\Theta_\beta(R_0)$, Lemma~\ref{lem:smallball} and the dimension prior give probability at least $w_n(\beta)e^{-C_1t}$ to
\[
 \alpha=\beta,\qquad J=J_*,\qquad
 \|\theta-\theta_0^{(J_*)}\|_2\le rJ_*^{-\beta}.
\]
On this event, $\theta\in\Theta_\beta(R_*)$ and
\[
 \|\theta-\theta_0\|_2
 \le rJ_*^{-\beta}+R_0(J_*+1)^{-\beta}
 \le(r+R_0)\eps.
\]
As $n\to\infty$, $\eps$ tends to zero uniformly over the grid. Lemma~\ref{lem:localKL} therefore bounds $K_{G_0}(p_0,p_\theta)+V_{G_0}(p_0,p_\theta)$ by $C\eps^2$ and the absolute log ratio by a fixed constant on this event, for all sufficiently large $n$. Apply Lemma~\ref{lem:denominator} with $u=1$ and $n\eps^2=t$. There is an event $\mathcal E_n$ satisfying
\begin{equation}\label{eq:oracleden}
 P_0^n(\mathcal E_n^c)\le e^{-c_Dt},\qquad
 D_n\ge w_n(\beta)e^{-C_Dt}
 \quad\text{on }\mathcal E_n.
\end{equation}
The constants $c_D,C_D$ are independent of $A$ because the smoothness probability is kept explicit. The prior weights also satisfy
\begin{equation}\label{eq:weightbounds}
 \frac{w_n(\alpha)}{w_n(\beta)}
 =e^{-A\{t_n(\alpha)-t\}},\qquad
 w_n(\beta)\ge L_n^{-1}e^{-At}.
\end{equation}

\medskip\noindent\emph{Step 2: Excluding excessively small smoothness indices.}
Choose $B>1$ so that $A(B-1)>C_D+3$, and define
\[
 \mathcal R_n=\{\alpha\in\cA_n:t_n(\alpha)>Bt\}.
\]
Use \eqref{eq:posterior-decomposition} with the zero test. Equations~\eqref{eq:oracleden}--\eqref{eq:weightbounds} give
\begin{align}\label{eq:roughdiscard}
 \E_0\Pi_n(\alpha\in\mathcal R_n\mid\cD_n)
 &\le e^{-c_Dt}
       +e^{C_Dt}\sum_{\alpha\in\mathcal R_n}
                      \frac{w_n(\alpha)}{w_n(\beta)}
       \notag\\
 &\le e^{-c_Dt}
       +L_n e^{\{C_D-A(B-1)\}t}
 \le Ce^{-ct}.
\end{align}
The last inequality holds uniformly because $\log L_n=o(n^{1/(2\alpha_++1)})$ and $t\ge n^{1/(2\alpha_++1)}$. Thus every fixed $A>0$ gives the required bound after choosing $B$.

\medskip\noindent\emph{Step 3: Covering numbers, tests, and contraction.}
Let $\gamma$ be the smallest grid point outside $\mathcal R_n$. Such a point exists because $\beta\notin\mathcal R_n$. We have $\gamma\le\beta$ and $t_n(\gamma)\le Bt$. By nesting of the coefficient ellipsoids, all retained densities belong to
\[
 \mathcal F_\gamma^*=\{p_\theta:\theta\in\Theta_\gamma(R_*)\}.
\]
For every $M\ge1$, Lemma~\ref{lem:entropy} gives
\begin{equation}\label{eq:retainedentropy}
 \log N(M\eps/64,\mathcal F_\gamma^*,d_{G_0})
 \le H_Bt,
\end{equation}
with $H_B$ depending on $B$ and the fixed class bounds. To verify its order, use
\[
 \eps^{-1/\gamma}
 =t_n(\gamma)
   \left\{\frac{t_n(\gamma)}{t}\right\}^{1/(2\gamma)}
 \le B^{1+1/(2\alpha_-)}t.
\]
All other constants in the covering bound are uniform over $\gamma\in[\alpha_-,\alpha_+]$; a radius exceeding the diameter of the class requires only one ball.

Lemma~\ref{lem:tests} and \eqref{eq:retainedentropy} give a test $\varphi_n$ such that
\begin{equation}\label{eq:maintests}
 P_0^n\varphi_n\le e^{(H_B-c_TM^2)t},\qquad
 \sup_{\substack{p\in\mathcal F_\gamma^*\\d_{G_0}(p,p_0)>M\eps}}
 P_p^n(1-\varphi_n)\le e^{-c_TM^2t}.
\end{equation}
Let $A_n^*$ be the set of prior parameters with $\alpha\notin\mathcal R_n$ and $d_{G_0}(p_\theta,p_0)>M\eps$. Substituting \eqref{eq:oracleden}, \eqref{eq:weightbounds}, and \eqref{eq:maintests} into \eqref{eq:posterior-decomposition} yields
\begin{equation}\label{eq:nonroughintegral}
 \E_0\Pi_n(A_n^*\mid\cD_n)
 \le e^{-c_Dt}+e^{(H_B-c_TM^2)t}
       +L_n e^{(C_D+A-c_TM^2)t}.
\end{equation}
Choose $M$ with $c_TM^2>H_B+C_D+A+3$. The right side is then bounded by $Ce^{-ct}$. Adding \eqref{eq:roughdiscard} proves contraction at rate $\eps$.

\medskip\noindent\emph{Step 4: Dimension and almost-sure convergence.}
The dimension prior gives $\Pi_n(J>L_Jt)\le e^{-\lambda L_Jt}$. Apply \eqref{eq:posterior-decomposition} again with the zero test:
\[
 \E_0\Pi_n(J>L_Jt\mid\cD_n)
 \le e^{-c_Dt}
       +L_n e^{(C_D+A-\lambda L_J)t}.
\]
Choose $\lambda L_J>C_D+A+3$. By \eqref{eq:gridcompare}, replace the contraction constant by $M\sqrt{C_\Delta}$ and take $L\ge C_\Delta\max(B,L_J)$. The union of the density-error and complexity events in \eqref{eq:mainbound} then has expected posterior probability at most $Ce^{-ct_n(\alpha_0)}$.

For a fixed truth, this expectation bound is summable because $t_n(\alpha_0)\ge n^{1/(2\alpha_++1)}$. Markov's inequality and Borel--Cantelli imply that the corresponding posterior probability tends to zero almost surely along the infinite iid sample.

For known smoothness, take $\gamma=\beta=\alpha_0$ and $w_n(\alpha_0)=1$. The second step is unnecessary; the other steps give the same contraction and dimension bounds. The dimension and coefficient priors are then independent of sample size.
\end{proof}

\begin{remark}[Posterior probability of smaller smoothness indices]\label{rem:nearest-grid}
A sufficiently large fixed $A$ gives a stronger complexity bound by excluding every grid point below $\beta$. If $\beta>\alpha_-$, the grid spacing implies
\[
 \frac{t_n(\beta-c_\Delta/\log n)}{t_n(\beta)}
 \ge\kappa_\Delta
 :=\exp\{2c_\Delta/(2\alpha_++1)^2\}>1.
\]
Every smaller grid point has at least this ratio. Repeating the calculation in \eqref{eq:roughdiscard} gives
\[
 \E_0\Pi_n(\alpha<\beta\mid\cD_n)
 \le e^{-c_Dt}
       +L_n e^{\{C_D-A(\kappa_\Delta-1)\}t}.
\]
Choose $A_0$ so that $A_0(\kappa_\Delta-1)>C_D+3$. There are $c>0$ and $n_0$, independent of $A\ge A_0$ and uniform over the theorem's truth and covariate classes, for which
\[
 \E_0\Pi_n(\alpha<\beta\mid\cD_n)
 \le2e^{-ct_n(\alpha_0)},\qquad n\ge n_0.
\]
If $\beta=\alpha_-$, the event is empty. Summability gives almost-sure convergence for each fixed truth. This conclusion controls the prior smoothness index relative to $\beta$; it does not identify a unique smoothness index for the density.
\end{remark}

\section{Proof of the minimax lower bound}\label{app:minimax}
\begin{proof}[Proof of Theorem~\ref{thm:minimax}]
Fix $\alpha\in[\alpha_-,\alpha_+]$ and $G_0\in\Gclass(\underline g,\overline g)$, and let $m=\lceil n^{1/(2\alpha+1)}\rceil$. We construct $2^m$ densities in the true class. Their coefficient distances give a lower bound on Hellinger separation, while neighboring sign vectors generate sample distributions with total variation bounded away from one.

\medskip\noindent\emph{Step 1: Density alternatives.}
For $\omega\in\{-1,1\}^m$, define
\[
 h_\omega
 =1+am^{-\alpha-1/2}\sum_{j=m+1}^{2m}\omega_{j-m}\phi_j,
 \qquad p_\omega=p_{h_\omega}.
\]
The constant $a>0$ will be chosen independently of $n,\alpha,G_0$. Every $h_\omega$ has outcome integral one. Its coefficient vector and supremum norm satisfy
\[
 \|\theta_\omega\|_{2,\alpha}^2
 \le a^2 2^{2\alpha_+},\qquad
 \|h_\omega-1\|_\infty
 \le am^{-\alpha}\{C_K(2m)^\rho\}^{1/2}\le Ca.
\]
The last inequality uses $\alpha>\rho/2$. Choose $a2^{\alpha_+}\le R_0$ and $Ca\le\min(1/2,1-b_0)$. Then every $p_\omega$ belongs to $\cF_\alpha(R_0,b_0)$, and $1/2\le h_\omega\le3/2$.

\medskip\noindent\emph{Step 2: Hellinger separation and neighboring tests.}
Let $d_H^{\rm bit}(\omega,\omega')$ be the number of differing coordinates. Proposition~\ref{prop:geometry}, using the common upper bound on the representations and $g_0\ge\underline g$, gives
\begin{equation}\label{eq:hypersep}
 d_{G_0}^2(p_\omega,p_{\omega'})
 \ge c_0\|\theta_\omega-\theta_{\omega'}\|_2^2
 =4c_0a^2m^{-2\alpha-1}d_H^{\rm bit}(\omega,\omega'),
\end{equation}
where $c_0>0$ is uniform over the stated classes. For sign vectors differing in one coordinate, the last assertion of Lemma~\ref{lem:localKL} gives
\[
 nK_{G_0}(p_\omega,p_{\omega'})
 \le Ca^2nm^{-2\alpha-1}\le Ca^2.
\]
Decrease $a$ until this is at most $1/4$.

For probability distributions $P,Q$, write $\TV(P,Q)=\sup_A|P(A)-Q(A)|$ and use the Hellinger normalization in \eqref{eq:hell}. Cauchy--Schwarz gives $\TV(P,Q)\le H(P,Q)$, and $-\log u\ge1-u$ gives $H^2(P,Q)\le\KL(P,Q)$. Independence gives the product identity
\[
 H^2(P^n,Q^n)
 =2-2\{1-H^2(P,Q)/2\}^n
 \le nH^2(P,Q).
\]
It follows that neighboring sample distributions have total variation distance at most $1/2$.

\medskip\noindent\emph{Step 3: From estimation to binary tests.}
Given a measurable conditional-density estimator $\widehat p_n$, select its nearest candidate $p_{\widehat\omega}$ in $d_{G_0}$, resolving ties by a fixed ordering. The selection is measurable because the candidate set is finite. By the triangle inequality and the definition of the nearest candidate,
\[
 d_{G_0}(p_{\widehat\omega},p_\omega)
 \le2d_{G_0}(\widehat p_n,p_\omega).
\]
Equation~\eqref{eq:hypersep} therefore yields
\[
 \E_\omega d_{G_0}^2(\widehat p_n,p_\omega)
 \ge c_0a^2m^{-2\alpha-1}
       \E_\omega d_H^{\rm bit}(\widehat\omega,\omega),
\]
where $\E_\omega$ denotes expectation under $P_{p_\omega,G_0}^n$.

Average over the $2^m$ sign vectors. For each coordinate, hold the other coordinates fixed and compare its two possible signs. The two testing-error probabilities sum to at least one minus the total variation distance of the neighboring sample distributions, hence to at least $1/2$. The average error probability for each coordinate is therefore at least $1/4$. Summing over coordinates gives
\[
 2^{-m}\sum_\omega
 \E_\omega d_{G_0}^2(\widehat p_n,p_\omega)
 \ge\frac{c_0a^2}{4}m^{-2\alpha}
 \ge c n^{-2\alpha/(2\alpha+1)}.
\]
The last constant is uniform because $m\le2n^{1/(2\alpha+1)}$ and $\alpha\le\alpha_+$. The supremum risk dominates this average. Taking the infimum over all estimators proves \eqref{eq:minimax}, including when the estimator knows $\alpha$ and $G_0$.
\end{proof}

\section{Functionals and polynomial smoothness}\label{app:extensions}
\subsection{Predictive risk}\label{app:predictive-details}
\begin{proof}[Proof of Corollary~\ref{cor:risk}]
Split the posterior integral according to whether $d_{G_0}(p_\theta,p_0)\le M\eps_n(\alpha_0)$. Theorem~\ref{thm:adaptive} and the bound $d_{G_0}^2\le2$ give
\[
 \E_0\int d_{G_0}^2(p_\theta,p_0)\,\Pi_n(d\theta\mid\cD_n)
 \le M^2\eps_n(\alpha_0)^2+2Ce^{-ct_n(\alpha_0)}
 \lesssim\eps_n(\alpha_0)^2.
\]
The comparison is uniform over the theorem's classes: $t_n(\alpha_0)$ grows at least as a fixed positive power of $n$, whereas the squared rates decrease polynomially.

Tonelli shows that the posterior mixture integrates to one over $y$ for almost every $x$. Concavity of the square root gives, for each sample,
\[
 \begin{split}
 d_{G_0}^2(\bar p_n,p_0)
 &=2-2\int\sqrt{p_0}
       \sqrt{\int p_\theta\,\Pi_n(d\theta\mid\cD_n)}
       \,d\mu_y\,dG_0\\
 &\le\int d_{G_0}^2(p_\theta,p_0)\,
                  \Pi_n(d\theta\mid\cD_n).
 \end{split}
\]
Taking $\E_0$ proves the predictive-risk bound. The argument does not require a common upper bound on the candidate densities and also applies to the later predictive-risk results.
\end{proof}

\subsection{Conditional functionals and uniform convergence}\label{app:functional-details}
\begin{proof}[Proof of Proposition~\ref{prop:functionals}]
For almost every $x$, Cauchy--Schwarz gives
\[
 \int|p-q|\,d\mu_y
 \le H_x\left\{\int(\sqrt p+\sqrt q)^2\,d\mu_y\right\}^{1/2}
 \le2H_x.
\]
The difference of two probabilities of the same event is at most half this $L_1$ distance. Apply that bound to $\{Y\le t\}$ and take the supremum over $t$ to obtain \eqref{eq:cdf}. Bounding the absolute integrand by $\|a\|_\infty|p-q|$ gives \eqref{eq:moment}.

Under the quantile assumptions, the conditional CDFs are continuous and $F_p(Q_p(\tau\mid x)\mid x)=F_0(Q_0(\tau\mid x)\mid x)=\tau$ for $0<\tau<1$. Also, $f_0\ge m$ implies
\[
 |F_0(v\mid x)-F_0(u\mid x)|\ge m|v-u|.
\]
Taking $u=Q_0(\tau\mid x)$ and $v=Q_p(\tau\mid x)$ gives
\[
 m|Q_p(\tau\mid x)-Q_0(\tau\mid x)|
 \le|F_0(Q_p(\tau\mid x)\mid x)-F_p(Q_p(\tau\mid x)\mid x)|.
\]
Taking the supremum over $\tau$ proves \eqref{eq:quantile}. Finally, integrating the conditional squared Hellinger distance against $dG_*/dG_0\le C_*$ proves \eqref{eq:overlap}.

The displayed losses are measurable. Right-continuity permits the CDF supremum to be taken over rational thresholds. Generalized quantiles are measurable by countable inversion of the conditional CDFs, and their left-continuity on $(0,1)$ permits the quantile supremum to be taken over rational indices. Integration of the bounded measurable moment integrand preserves measurability. Squaring the bounds and integrating over $G_0$, the posterior, and $P_0^n$ proves the stated risk consequences.
\end{proof}

\begin{proof}[Proof of Theorem~\ref{thm:uniform}]
The proof first excludes nonpositive representations. The reverse Hellinger inequality then gives coefficient contraction, and interpolation gives uniform convergence.

\medskip\noindent\emph{Step 1: Positivity and coefficient recovery.}
Lemma~\ref{lem:interpolation} and Assumption~\ref{ass:compact} imply that
\[
 \mathcal H_*=\{h_\theta:\theta\in\Theta_{\alpha_-}(R_*)\},
 \qquad
 \mathcal H_{\rm true}
 =\{h_\theta:\theta\in\Theta_{\alpha_-}(R_0),\ \min h_\theta\ge b_0\}
\]
are compact in the uniform norm. Their elements are continuous and uniformly bounded, with outcome integral one at every $x$. Since their normalizers are at least one, the map $h\mapsto p_h$ is continuous in the uniform norm on these classes. The set $\mathcal B=\{h\in\mathcal H_*:\min h\le0\}$ is closed in $\mathcal H_*$ and hence compact.

The density images of $\mathcal B$ and $\mathcal H_{\rm true}$ are disjoint. If two such densities agreed almost everywhere, continuity and full support of $\mu$ would make them equal everywhere. The density generated by the true representation is strictly positive, so the other representation could not vanish. For each $x$, continuity and connectedness of $\mathcal Y$ would make its sign constant in $y$. Its outcome integral one would force that sign to be positive, contradicting membership in $\mathcal B$.

The two density images are compact, and zero Hellinger distance under $\mu$ would imply equality everywhere. If $\mathcal B$ is nonempty, their minimum Hellinger distance under $\mu$ is therefore strictly positive. If it is empty, there is nothing to exclude. The lower bound $g_0\ge\underline g$ makes the positive separation uniform over $G_0$. For all sufficiently large $n$, Theorem~\ref{thm:adaptive} thus gives
\[
 \E_0\Pi_n(h_\theta\in\mathcal B\mid\cD_n)
 \le Ce^{-ct_n(\alpha_0)}.
\]
Outside $\mathcal B$, both $h_\theta$ and $h_0$ are nonnegative and have a common upper bound. Proposition~\ref{prop:geometry} and orthonormality give $\|\theta-\theta_0\|_2\le Cd_{G_0}(p_\theta,p_0)$. Together with Theorem~\ref{thm:adaptive}, this proves the coefficient rate.

\medskip\noindent\emph{Step 2: Uniform convergence and risk.}
On the contraction and complexity event of Theorem~\ref{thm:adaptive}, put $\gamma=\min(\alpha,\alpha_0)$. If $\alpha<\alpha_0$, the bound $t_n(\alpha)\le Lt_n(\alpha_0)$ gives
\[
 \frac{2(\alpha_0-\alpha)\log n}
      {(2\alpha+1)(2\alpha_0+1)}
 =\log\{t_n(\alpha)/t_n(\alpha_0)\}\le\log L.
\]
Hence $0\le\alpha_0-\gamma\le C/\log n$. Since $\|\theta-\theta_0\|_{2,\gamma}\le R_*+R_0$, interpolation on the coefficient-contraction event yields
\[
 \|h_\theta-h_0\|_\infty
 \le C\eps_n(\alpha_0)^{1-\rho/(2\gamma)}
 \le C u_n(\alpha_0).
\]
The second inequality uses $(\alpha_0-\gamma)\log n\le C$ and $\gamma\ge\alpha_->\rho/2$.

For any two representations with outcome integral one and $|h|,|k|\le B$,
\[
 \|Q_h-Q_k\|_\infty\le2B\|h-k\|_\infty,\qquad
 \|p_h-p_k\|_\infty
 \le2B(1+B^2)\|h-k\|_\infty.
\]
These follow by expanding $h^2-k^2$ and using $Q_h,Q_k\ge1$. They give the asserted uniform density rate. Since $u_n(\alpha_0)\to0$ uniformly over the smoothness interval, the representation error is eventually smaller than $b_0/2$. Thus $\min h_\theta>b_0/2$ on the same event, proving \eqref{eq:uniform-contraction}.

The coefficient, representation, and density losses are uniformly bounded over the prior and truth classes. Split each posterior expectation at its established rate and use the exponentially small probability of the complement. This gives squared coefficient risk $O(\eps_n(\alpha_0)^2)$ and squared uniform risks $O(u_n(\alpha_0)^2)$.
\end{proof}

\subsection{Polynomial bases and differentiability}\label{app:polynomial-details}
\begin{proof}[Proof of Proposition~\ref{prop:classical}]
Under $dt/2$, the normalized Legendre polynomials are $\psi_k=\sqrt{2k+1}P_k$. The identities $P_k(1)=1$ and $|P_k(t)|\le1$ for $t\in[-1,1]$ imply
\[
 |\Phi_\nu(z)|^2\le\prod_{r=1}^q(2\nu_r+1)
 =|\Phi_\nu(z_*)|^2;
\]
see \citet[Sections 18.3 and 18.14]{DLMF}. Enlarging the terms of total degree at most $M$ to the coordinate rectangle $0\le\nu_r\le M$ bounds their squared basis sum by
\[
 \prod_{r=1}^q\sum_{k=0}^M(2k+1)=(M+1)^{2q}.
\]
For a lower bound at $z_*$, retain multi-indices with every coordinate between $M/(4q)$ and $M/(2q)$. For large $M$, there are at least $cM^q$ such indices and each contributes at least $cM^q$. They have positive outcome degree and total degree at most $M$. Thus the lower and upper sums are both of order $M^{2q}\asymp J_M^2$. An incomplete degree level contains all terms from the preceding level and is contained in the next complete level. Their counts are comparable, proving \eqref{eq:legendre-sharp} after adjusting constants for small $J$.

Under the arcsine reference measure, $\psi_0=1$ and $\psi_k=\sqrt2T_k$ for $k\ge1$. The identity $T_k(\cos u)=\cos(ku)$ verifies orthonormality and gives $|\psi_k|\le\sqrt2$. Each tensor term has squared magnitude at most $2^q$ everywhere and at least one at $z_*$. Summation proves \eqref{eq:cheb-sharp}.

Finally, subtracting the multi-indices with zero outcome degree from all indices of total degree at most $M$ gives \eqref{eq:count}. Taking the difference between successive counts gives the number of terms of degree exactly $M$.
\end{proof}

\begin{proof}[Proof of Corollary~\ref{cor:physical}]
The degree-index norm comparison embeds the degree class in an index ellipsoid with $\alpha=s/q$. It also embeds an index ellipsoid of a smaller positive radius in the degree class. The radius comparisons are uniform over the prescribed compact interval of $s$. Proposition~\ref{prop:classical} gives $\rho=2$ for Legendre and $\rho=1$ for Chebyshev. With $g_0\le\overline g$, apply Theorem~\ref{thm:adaptive} and Corollary~\ref{cor:risk} to the containing ellipsoid. If also $g_0\ge\underline g$, apply Theorem~\ref{thm:minimax} to the contained ellipsoid. Both inclusions retain $h\ge b_0$. Substituting $\alpha=s/q$ gives \eqref{eq:physical-rate} and its square as the minimax risk.
\end{proof}

\begin{proof}[Proof of Proposition~\ref{prop:embedding-sharp}]
Lemma~\ref{lem:interpolation} proves boundedness and uniform convergence when $\alpha>\rho/2$. To show necessity, define
\[
 A_N=\sum_{j\le N}j^{-2\alpha}\phi_j(z_*)^2,\qquad
 \theta_j^{[N]}
 =A_N^{-1/2}j^{-2\alpha}\phi_j(z_*)\ind\{j\le N\}.
\]
The finite vector $\theta^{[N]}$ has weighted norm one, and its series value at $z_*$ equals $A_N^{1/2}$. A uniform bound would therefore require $\sup_NA_N<\infty$.

By Proposition~\ref{prop:classical}, $S_m=\sum_{j\le m}\phi_j(z_*)^2\asymp m^\rho$. Summation by parts gives
\[
 A_N=N^{-2\alpha}S_N
 +\sum_{m=1}^{N-1}
       \{m^{-2\alpha}-(m+1)^{-2\alpha}\}S_m.
\]
For $\alpha>0$ and $2\alpha\le\rho$, the sum diverges because its summands have order $m^{\rho-2\alpha-1}$; divergence is logarithmic at equality. For $\alpha=0$, $A_N=S_N\to\infty$. Thus the thresholds are $\alpha>1$ for Legendre and $\alpha>1/2$ for Chebyshev.

Every finite series is continuous, and the reference measures have full support. Its essential supremum therefore equals its supremum, so the same examples rule out a uniform essential-supremum bound. Above the threshold, linearity and the interpolation bound give the stated bounded linear map.
\end{proof}

\begin{proof}[Proof of Proposition~\ref{prop:sobolev}]
For every multi-index $\nu$,
\[
 1+\lambda_\nu
 \le(1+|\nu|_1)^2
 \le(q+1)(1+\lambda_\nu).
\]
Raise these inequalities to $s$, multiply by the squared Legendre coefficients of $v$, and sum. This proves \eqref{eq:spectral-norm} and equivalence of the two domains. For a polynomial $v$, integration by parts in each coordinate gives
\[
 \langle v,\mathcal Lv\rangle
 =\sum_{r=1}^q\int(1-z_r^2)|\partial_{z_r}v|^2\,d\mu.
\]
The boundary terms vanish because $1-z_r^2=0$ at $z_r=\pm1$. Adding $\|v\|_2^2$ proves \eqref{eq:weighted-energy}.

The identity extends by completing polynomials in the norm given by the square root of its right side. In coefficients, this squared norm is $\sum_\nu(1+\lambda_\nu)\theta_\nu^2$. Finite sequences are dense in this weighted coefficient space, whose image in $L_2(\mu)$ is $D((I+\mathcal L)^{1/2})$. The quantities $\sqrt{1-z_r^2}\,\partial_{z_r}v$ in the completion are the $L_2(\mu)$ limits of the corresponding polynomial derivatives. This specifies the derivative space needed for the identity without claiming an identification with a maximal weak-derivative domain.

For the final assertion, positivity on the compact cube and the $C^{2m}$ extension imply that $\sqrt p$ is $C^{2m}$ on a neighborhood of the cube. Its derivatives are bounded on a sufficiently small compact neighborhood. Differentiation under the outcome integral is therefore justified by dominated convergence and gives $c_p\in C^{2m}$. The integral $c_p$ has a positive lower bound, so $v=h_p-1=\sqrt p/c_p-1$ is also $C^{2m}$.

The differential expression $(I+\mathcal L)^mv$ is continuous on the cube and hence belongs to $L_2(\mu)$. Repeated integration by parts against tensor Legendre polynomials gives its coefficient at $\nu$ as $(1+\lambda_\nu)^m\theta_\nu$. At each integration, the factor $1-z_r^2$ removes the boundary term. Parseval and \eqref{eq:spectral-norm} now give
\[
 \|h_p-1\|_{2m,\mathrm{deg}}^2
 \le(q+1)^{2m}
       \sum_\nu(1+\lambda_\nu)^{2m}\theta_\nu^2
 =(q+1)^{2m}\|(I+\mathcal L)^mv\|_2^2<\infty.
\]
\end{proof}

\begin{proof}[Proof of Corollary~\ref{cor:anisotropic}]
The inequality $\Lambda_\nu\le T$ defines a coordinate rectangle with order $T^{D_a}$ indices. Removing indices with zero outcome degree removes a fraction of order $T^{-1/a_1}$ and leaves the same order. Hence the $j$th ordered weight is comparable to $j^{1/D_a}$. The anisotropic coefficient norm is therefore equivalent to the index norm with $\alpha=u/D_a$, with comparison constants uniform on the prescribed compact interval of $u$.

For Legendre terms, the squared basis sum over this rectangle is bounded by $C\prod_rT^{2/a_r}=CT^{2D_a}$. This also gives a bound of order $j^2$ for an incomplete level: if the $j$th weight is $T$, all retained terms lie in the rectangle at $T$, while the smaller rectangle at $T/2$ precedes that level and has at least $cT^{D_a}$ terms for large $T$. Thus $T^{D_a}\lesssim j$. Adjusting the constant for small $j$ verifies \eqref{eq:kernel} with $\rho=2$. For Chebyshev, each tensor term is bounded in absolute value by $2^{q/2}$, giving $\rho=1$ under this ordering as well.

After the norm and radius comparisons, Theorem~\ref{thm:adaptive} and Corollary~\ref{cor:risk} apply under $g_0\le\overline g$. The additional bound $g_0\ge\underline g$ permits application of Theorem~\ref{thm:minimax}. Substituting $\alpha=u/D_a$ gives \eqref{eq:anisotropic-rate}, its squared minimax risk, and the stated smoothness thresholds.
\end{proof}

\subsection{Sparse approximation}\label{app:selection-details}
\begin{proof}[Proof of Corollary~\ref{cor:bestk}]
Put $s_n=|S_n|$. Since $S_n$ is admissible, it can be used in the minimum defining \eqref{eq:sparseoracle}. The inequalities $\binom{K_n}{s_n}\le K_n^{s_n}$ and $K_n\le n^\kappa$ give
\[
 \eta_n^2
 \le R^2k_n^{-2\gamma}
       +\{(\kappa+1)s_n+1\}\frac{\log n}{n}.
\]
Use $k_n\ge aI_n$, $s_n\le k_n\le bI_n$, and $I_n\log n/n=I_n^{-2\gamma}$. Since $I_n\ge1$ eventually, $\log n/n\le I_n^{-2\gamma}$. Therefore
\[
 \eta_n^2
 \le\{R^2a^{-2\gamma}+(\kappa+1)b+1\}I_n^{-2\gamma}.
\]
Taking square roots gives the stated rate. It tends to zero, so the condition $\eta_n\le1$ in Theorem~\ref{thm:sparse} holds for all sufficiently large $n$.
\end{proof}

\section{Proof for densities that can vanish}\label{app:floor}
Write $H=d_{G_0}$. The constant one in the mixtures below is the reference density with respect to $\mu_y$. Joint convexity of squared Hellinger distance gives, for $0\le\delta\le1$,
\begin{align}
 H^2((1-\delta)p+\delta,(1-\delta)q+\delta)&\le H^2(p,q),\label{eq:mix-contract}\\
 H^2(p,(1-\delta)p+\delta)&\le2\delta.\label{eq:mix-bias}
\end{align}
For the first bound, integrate the inequality
\[
 \sqrt{\{(1-\delta)p+\delta\}\{(1-\delta)q+\delta\}}
 \ge(1-\delta)\sqrt{pq}+\delta.
\]
For the second, write $p=(1-\delta)p+\delta p$ and use $H^2(p,1)\le2$. The triangle inequality and Proposition~\ref{prop:geometry} therefore imply
\begin{equation}\label{eq:floor-approx}
 H^2(p_0,p_{1,\theta,n})
 \le2\overline g\norm{\theta-\theta_0}_2^2+4\delta_n.
\end{equation}
This approximation bound does not require the representations to be nonnegative.

\begin{lemma}[Likelihood ratios when the truth can vanish]\label{lem:upper-ratio}
Let $p,q$ be conditional densities with $q>0$ almost everywhere under $\mu_y\otimes G_0$. Suppose $p/q\le M$ almost surely under $P_{p,G_0}$, for some finite $M\ge1$. Then
\begin{equation}\label{eq:upper-ratio}
 K_{G_0}(p,q)+V_{G_0}(p,q)
 \le C(1+\log M)^2H^2(p,q),\qquad
 \log(p/q)\le\log M\quad P_{p,G_0}\text{-almost surely},
\end{equation}
where $C$ is a numerical constant.
\end{lemma}
\begin{proof}
On $A=\{p>0\}$, put $t=q/p$. The ratio bound gives $t\ge1/M$. For $t\in[1/M,\infty)$,
\[
 \frac{t-1-\log t}{(\sqrt t-1)^2}\le C(1+\log M),\qquad
 \frac{(\log t)^2}{(\sqrt t-1)^2}\le C(1+\log M)^2.
\]
Both ratios extend continuously at one and are bounded on $[1/2,2]$. For $t\ge2$, the denominator is bounded below by a positive numerical multiple of $t$, which bounds both ratios. For $1/M\le t\le1/2$, the denominator is bounded below by a positive numerical constant, and $|\log t|\le\log M$.

All integrals in the rest of the proof are with respect to $\mu_y\otimes G_0$. Multiplying the second bound by $p$ and integrating over $A$ bounds $V_{G_0}(p,q)$ by the right-hand side of \eqref{eq:upper-ratio}. In particular, the log ratio is integrable under $P_{p,G_0}$. Normalization of $p$ and $q$ then gives
\[
 K_{G_0}(p,q)
 =\int_A p(t-1-\log t)+\int_{A^c}q.
\]
The first scalar bound controls the integral over $A$. On $A^c$, the squared-Hellinger integrand equals $q$, whereas the integrand defining $V_{G_0}$ is zero. Combining these bounds proves the moment inequality. Taking logarithms of $p/q\le M$ proves the almost-sure upper bound.
\end{proof}

\begin{proof}[Proof of Theorem~\ref{thm:floor}]
We first construct a prior event that bounds the posterior denominator. We then use prior weight ratios and tests to control the remaining posterior probability. Fix $\alpha_0,p_0,G_0$ in the classes of the theorem. Constants in the first assertion are uniform over these choices and may depend on the fixed smoothness endpoints, radii, $B_0$, $\overline g$, and prior parameters. The proof of the positive-subclass assertion also uses $b_0$, $\alpha_p$, and the constants in Assumption~\ref{ass:kernel}.

\textit{1. Approximation and prior probability.} Let $\beta$ be the largest grid point not exceeding $\alpha_0$. For the first assertion, set
\[
 t=t_{1,n}(\beta),\qquad \eps=\sqrt{t/n},\qquad J_* =\lceil t\rceil.
\]
Fix $r\in(0,R_*-R_0)$ and consider the prior event
\[
 \ell=1,\qquad \alpha=\beta,\qquad J=J_*,\qquad
 \norm{\theta-\theta_0^{(J_*)}}_2\le rJ_*^{-\beta}.
\]
Since $\theta_0\in\Theta_\beta(R_0)$, Lemma~\ref{lem:smallball} and the dimension prior give this event probability at least $w_n(1,\beta)e^{-C_1t}$. Truncation and the radius of the coefficient ball give
\[
 \norm{\theta-\theta_0}_2
 \le rJ_*^{-\beta}+R_0(J_*+1)^{-\beta}
 \le CJ_*^{-\beta}.
\]

The outcome integral of $h_0$ is one, so $Q_0\ge1$. Hence $|h_0|\le B_0$ implies $p_0\le B_0^2$, whether $h_0$ is signed or nonnegative. Every density in the approximating event satisfies $p_{1,\theta,n}\ge\delta_n$. Applying Lemma~\ref{lem:upper-ratio} with $M=B_0^2/\delta_n$ and using \eqref{eq:floor-approx} yields
\[
 \begin{split}
 K_{G_0}(p_0,p_{1,\theta,n})+V_{G_0}(p_0,p_{1,\theta,n})
 &\le C(\log n)^2\{J_*^{-2\beta}+\delta_n\}
 \le C\eps^2,\\
 \log(p_0/p_{1,\theta,n})&\le C\log n
 \quad P_0\text{-almost surely}.
 \end{split}
\]
The last moment bound follows from $t^{2\beta+1}=n(\log n)^2$ and $\delta_n=n^{-3}$. It is uniform over the fixed smoothness interval.

\textit{2. The posterior denominator.} For a prior event of probability $b_n$ on which $K\le C_K'\eps^2$, $V\le C_V\eps^2$, and $\log(p_0/p)\le\Lambda_n$ almost surely under $P_0$, Lemma~\ref{lem:denominator} with margin one gives
\begin{equation}\label{eq:one-sided-denom}
 P_0^n\{D_n<b_n e^{-(C_K'+1)n\eps^2}\}
 \le\exp\left\{-\frac{n\eps^2}{2(C_V+\Lambda_n/3)}\right\}.
\end{equation}
Apply this inequality to the event just constructed. Since $n\eps^2=t$ and $\Lambda_n\le C\log n$, there is an event $\mathcal E_n$ such that
\begin{equation}\label{eq:floor-oracle-den}
 P_0^n(\mathcal E_n^c)\le e^{-ct/\log n},\qquad
 D_n\ge w_n(1,\beta)e^{-C_Dt}
 \quad\text{on }\mathcal E_n.
\end{equation}
The constant $C_D$ does not depend on $A$, because the component weight has not been included in the exponential constant.

For the positive subclass, use the component $\ell=0$ and instead set $t=t_{0,n}(\beta)$, $\eps=\sqrt{t/n}$, and $J_*=\lceil t\rceil$. Use the same coefficient ball. Since $\alpha_0\ge\alpha_p>\rho/2$, the grid point $\beta$ eventually lies in a fixed interval above $\rho/2$, uniformly in $\alpha_0$. Lemma~\ref{lem:localKL} then gives $K+V\le C\eps^2$ and a fixed upper bound on the log ratio. Thus \eqref{eq:floor-oracle-den} holds with $w_n(0,\beta)$ and failure probability $e^{-ct}$.

\textit{3. Prior weights, entropy, and tests.} In either case, let $(\ell_*,\beta)$ denote the approximating component and retain $t=t_{\ell_*,n}(\beta)$ and $\eps=\sqrt{t/n}$. The component weights satisfy
\[
 w_n(\ell_*,\beta)\ge(2L_n)^{-1}e^{-At},\qquad
 \frac{w_n(\ell,\alpha)}{w_n(\ell_*,\beta)}
 =e^{-A\{t_{\ell,n}(\alpha)-t\}}.
\]
Choose $B>1$ so that $A(B-1)>C_D+3$. Equation~\eqref{eq:posterior-decomposition}, applied with the zero test, bounds the expected posterior probability of components with $t_{\ell,n}(\alpha)>Bt$ by
\[
 P_0^n(\mathcal E_n^c)
 +e^{C_Dt}\sum_{t_{\ell,n}(\alpha)>Bt}
       \frac{w_n(\ell,\alpha)}{w_n(\ell_*,\beta)}
 \le P_0^n(\mathcal E_n^c)
       +2L_n e^{\{C_D-A(B-1)\}t}.
\]
This argument works for every fixed $A>0$ by increasing $B$ as needed.

Let $\gamma$ be the smallest smoothness index among the retained components. Their coefficients belong to $\Theta_\gamma(R_*)$, and $t_{0,n}(\gamma)\le Bt$. The entropy bound requires only the upper bound on $g_0$. Specifically, Proposition~\ref{prop:geometry}, \eqref{eq:mix-contract}, and Lemma~\ref{lem:entropy} apply to each of the two density specifications. The identity
\[
 \eps^{-1/\gamma}
 =t_{0,n}(\gamma)
   \left\{\frac{t_{0,n}(\gamma)}{t}\right\}^{1/(2\gamma)}
 \le B^{1+1/(2\alpha_-)}t
\]
shows that the log covering number at radius $M\eps/64$, for $M\ge1$, is at most $C_Bt$. Combining the two coverings adds at most $\log2$, which can be included in $C_Bt$.

Lemma~\ref{lem:tests} therefore gives a test with type-I error at most $e^{(C_B-c_TM^2)t}$ and type-II error at most $e^{-c_TM^2t}$ against retained densities at distance greater than $M\eps$ from $p_0$. In \eqref{eq:posterior-decomposition}, the integrated type-II term divided by the denominator lower bound is at most
\[
 2L_n e^{(C_D+A-c_TM^2)t}.
\]
Choose $M$ so that $c_TM^2>C_B+C_D+A+3$, and add the bound for discarded components. Since $\log L_n=o(t)$ uniformly, the resulting expected posterior probability is at most $Ce^{-ct/\log n}$ for the bounded class and $Ce^{-ct}$ for the positive subclass.

\textit{4. Rates and risk.} The grid spacing gives $0\le\alpha_0-\beta<c_\Delta/\log n$. Consequently, $t_{\ell_*,n}(\beta)\asymp t_{\ell_*,n}(\alpha_0)$ uniformly over the prescribed smoothness interval. For $\ell_*=1$, use $\log\{n(\log n)^2\}/\log n=O(1)$ in this comparison. Increasing $M$ by the resulting fixed factor proves \eqref{eq:floor-bound} and \eqref{eq:mixed-exact-bound}.

In both cases, the exponential remainder is smaller than every fixed negative power of $n$, uniformly over the relevant class. Since squared Hellinger loss is at most two, splitting its posterior expectation at $M^2\eps^2$ gives the stated squared-risk bounds. The predictive bounds follow from convexity as in Appendix~\ref{app:predictive-details}.

For the prior using only the regularized specification, retain $\ell=1$ throughout. The coefficient probability, approximation, and entropy bounds are unchanged, and $w_n^{\mathrm{floor}}$ has the same weight-ratio formula. The argument therefore proves its contraction and risk conclusions. None of the first-assertion calculations uses completeness, a basis-growth bound, or a lower bound on $g_0$.
\end{proof}

\section{Minimax risk at every positive degree of smoothness}\label{app:all-s}
\begin{proof}[Proof of Theorem~\ref{thm:all-s-minimax}]
We construct an estimator for the upper bound and bounded positive alternatives for the lower bound. The constants are uniform in $n$, $s\in[s_-,s_+]$, and $G_0\in\Gclass(\underline g,\overline g)$; they may depend on $q$, the fixed radii and bounds, and the smoothness endpoints. Both covariate-density bounds are used for the finite-candidate estimator below, although its construction does not use their values.

\textit{1. An estimator based on a finite set of densities.} Let $q_1,\ldots,q_N$ be conditional densities chosen before observing the data. For each ordered pair $(i,j)$, compute the sum of the scores in Lemma~\ref{lem:tests}, with $q_i$ as the first density and $q_j$ as the second. Say that $j$ is preferred to $i$ if this sum is nonnegative. At zero, both comparisons are accepted. The type-II bound for the reversed comparison controls the non-strict error event in this rule.

Define
\[
 D_{ij}=d_{\mu_x}^2(q_i,q_j),\qquad
 S_i=\max\bigl(\{D_{ij}:j\text{ is preferred to }i\}\cup\{0\}\bigr),
\]
and let $\widehat i$ minimize $S_i$, resolving ties by the smallest index. The scores and the reference-measure distances $D_{ij}$ require neither $G_0$ nor its density bounds. The estimator is measurable because the candidate set is finite and all scores are measurable. For every $i$,
\[
 d_{\mu_x}^2(q_{\widehat i},q_i)\le S_i.
\]
Indeed, one of $i$ and $\widehat i$ is preferred to the other. Their squared distance therefore enters either $S_i$ or $S_{\widehat i}\le S_i$.

Fix $i$ and set $e_i=d_{\mu_x}^2(p_0,q_i)$. If $D_{ij}>1024(\overline g/\underline g)e_i$, then
\[
 d_{G_0}(p_0,q_i)
 \le\sqrt{\overline g e_i}
 <\frac{\sqrt{\underline g D_{ij}}}{32}
 \le\frac{d_{G_0}(q_i,q_j)}{32}.
\]
Apply the type-II bound in Lemma~\ref{lem:tests} with $q_i$ as the neighborhood center and $q_j$ as the null density. Writing $a=c_Tn\underline g$, we obtain
\[
 P_0^n\{j\text{ is preferred to }i\}\le e^{-aD_{ij}}
 \quad\text{when }D_{ij}>1024(\overline g/\underline g)e_i.
\]
For $t\ge1024(\overline g/\underline g)e_i$,
\[
 S_i\le t+\sum_j(D_{ij}-t)_+
             \ind_{\{j\text{ is preferred to }i\}},
 \qquad
 \E_0 S_i\le t+Na^{-1}e^{-at}.
\]
The expectation bound follows from $(D_{ij}-t)e^{-aD_{ij}}\le a^{-1}e^{-at}$ when $D_{ij}>t$. Taking $t=1024(\overline g/\underline g)e_i+\log(N)/a$ and using $d_{G_0}^2\le\overline g d_{\mu_x}^2$ gives
\begin{equation}\label{eq:finite-tournament-oracle}
 \E_0d_{G_0}^2(q_{\widehat i},p_0)
 \le C_{\underline g,\overline g}
 \left\{d_{\mu_x}^2(p_0,q_i)+\frac{\log N+1}{n}\right\},
 \qquad i=1,\ldots,N.
\end{equation}
Here we used the squared triangle inequality to bound the loss by $2\overline g(S_i+e_i)$.

Proposition~\ref{prop:geometry} and Lemma~\ref{lem:entropy} give a finite covering of the density image of $\Theta_\alpha(R)$ with $\log N\le C\eps^{-1/\alpha}$ in $d_{\mu_x}$. The coefficient net, and hence the candidate densities, can be chosen without $G_0$ or its bounds. At $\eps=n^{-\alpha/(2\alpha+1)}$, this covering satisfies $\log N\le Cn\eps^2$ and $\min_i e_i\le C\eps^2$. Equation~\eqref{eq:finite-tournament-oracle} therefore gives risk at most $C\eps^2$ for every mean-one $L_2(\mu)$ representation in the ellipsoid. Boundedness and nonnegativity of the representation are not needed for this upper bound. The degree-index norm comparison with $\alpha=s/q$ proves the upper bound in \eqref{eq:all-s-minimax} for both polynomial systems.

\textit{2. Lower bound for Legendre series.} Choose a fixed even $C^\infty$ function $\eta$, supported in $[-1/4,1/4]$, equal to one on $[-1/8,1/8]$, and satisfying $0\le\eta\le1$. Put
\[
 b(z)=z_1\prod_{r=1}^q\eta(z_r).
\]
The function $b$ has positive squared $L_2$ norm and zero integral in its first coordinate. Partition $[-1/2,1/2]^q$ into $M^q$ equal cubes with centers $z_k$, and define
\[
 b_{M,k}(z)=b(M(z-z_k)),\qquad
 v_\omega=aM^{-s}\sum_k\omega_k b_{M,k},\qquad
 h_\omega=1+v_\omega,
 \quad\omega\in\{-1,1\}^{M^q}.
\]
Each bump is supported inside its cube, so the supports are disjoint. Oddness in the outcome coordinate gives $\int h_\omega(y,x)\,\mu_y(dy)=1$ for every $x$. A change of variables yields
\begin{align}
 \norm{v_\omega}_\infty&\le a\norm{b}_\infty,\label{eq:bump-bounded}\\
 \norm{D^\zeta v_\omega}_2^2
 &\le C_\zeta a^2M^{2|\zeta|_1-2s}
 \quad\text{for every fixed derivative multi-index }\zeta.
 \label{eq:bump-derivatives}
\end{align}
In the second bound, each differentiated bump contributes a factor $M^{2|\zeta|_1-q}$ and the $M^q$ supports are disjoint.

Fix an integer $r>s_+/2$. The differential expression $(I+\mathcal L)^r$ is a sum of partial derivatives of orders at most $2r$, with bounded smooth coefficients on the cube. Thus
\[
 \norm{v_\omega}_2^2\le Ca^2M^{-2s},\qquad
 \norm{(I+\mathcal L)^rv_\omega}_2^2\le Ca^2M^{4r-2s}.
\]
The bumps are smooth and supported in the interior. Repeated integration by parts therefore shows that the coefficient of $(I+\mathcal L)^rv_\omega$ at index $\nu$ is $(1+\lambda_\nu)^r\theta_{\omega,\nu}$, without boundary terms. Parseval's identity and H\"older's inequality give
\begin{align*}
 \sum_\nu(1+\lambda_\nu)^s\theta_{\omega,\nu}^2
 &\le\left(\sum_\nu\theta_{\omega,\nu}^2\right)^{1-s/(2r)}
       \left(\sum_\nu(1+\lambda_\nu)^{2r}
                       \theta_{\omega,\nu}^2\right)^{s/(2r)}\\
 &\le Ca^2.
\end{align*}
The powers of $M$ cancel. The coefficients with outcome degree zero vanish because the outcome integrals are zero. Proposition~\ref{prop:sobolev} now gives $\norm{v_\omega}_{s,\mathrm{deg}}\le Ca$, with $C$ uniform for $s\in[s_-,s_+]$. Choose $a>0$ so that this norm is at most $R$ and $a\norm{b}_\infty\le\min(1/2,B-1)$. Then $1/2\le h_\omega\le\min(B,3/2)$ and every $p_\omega=p_{h_\omega}$ belongs to the stated class.

For two sign vectors, disjoint supports imply
\[
 \norm{h_\omega-h_{\omega'}}_2^2
 =c_ba^2M^{-2s-q}d_H^{\mathrm{bit}}(\omega,\omega'),
 \qquad c_b>0,
\]
where $d_H^{\mathrm{bit}}$ counts their differing coordinates. The lower bound on $g_0$ and the inverse inequality in Proposition~\ref{prop:geometry} give the same lower bound, up to a fixed constant, for squared Hellinger distance. For neighbors differing in one coordinate, the upper comparison gives
\[
 n d_{G_0}^2(p_\omega,p_{\omega^{(k)}})
 \le Ca^2nM^{-2s-q}.
\]
Take $M=\lceil n^{1/(2s+q)}\rceil$ and decrease $a$, independently of $n$ and $s$, so that the right-hand side is at most $1/4$. The product-Hellinger and total-variation inequalities in Appendix~\ref{app:minimax} then bound the total variation distance between neighboring sample distributions by $1/2$.

Given any conditional-density estimator, select its nearest alternative in $d_{G_0}$ with fixed tie breaking. The same binary-testing argument as in Appendix~\ref{app:minimax} gives average Hamming error at least $M^q/4$. The triangle inequality and the preceding separation bound therefore imply worst-case squared risk at least
\[
 ca^2M^{-2s-q}M^q
 =ca^2M^{-2s}
 \ge c'n^{-2s/(2s+q)}.
\]
The last constant is uniform because $M\le2n^{1/(2s+q)}$ and $s\le s_+$.

\textit{3. Lower bound for Chebyshev series.} Let $U$ be uniform probability on $[-1,1]$ and define
\[
 F(u)=\cos\{\pi(u+1)/2\},\qquad
 I(z)=2\arccos(z)/\pi-1.
\]
If $u$ has distribution $U$, then $F(u)$ has the arcsine reference distribution; $I$ is the inverse transformation. Coordinatewise transformation of the tensor Chebyshev basis gives
\[
 \widetilde\Phi_\nu(u)=\Phi_\nu(F(u))
 =\prod_{r=1}^q c_{\nu_r}\cos\{\nu_r\pi(u_r+1)/2\},
 \qquad c_0=1,\quad c_k=\sqrt2\ (k\ge1).
\]
Use the preceding bumps in the uniform coordinates and set $h_\omega^{\mathrm C}(z)=1+v_\omega(I(z))$. The transformation preserves the outcome normalization, uniform bounds, disjoint supports, and reference $L_2$ distances. The Chebyshev coefficients satisfy
\[
 \theta_{\omega,\nu}
 =\int v_\omega(u)\widetilde\Phi_\nu(u)\,
                   U^{\otimes q}(du).
\]

Choose an integer $m$ with $2m>s_+$. The bumps and their derivatives vanish near the boundary. For each coordinate $r$, integration by parts $2m$ times gives
\[
 \int\partial_{u_r}^{2m}v_\omega(u)\widetilde\Phi_\nu(u)
                   \,U^{\otimes q}(du)
 =(-1)^m(\nu_r\pi/2)^{2m}\theta_{\omega,\nu}.
\]
Using $(1+|\nu|_1)^{4m}\le C(1+\sum_r\nu_r^{4m})$, Bessel's inequality, and \eqref{eq:bump-derivatives}, we obtain
\[
 \begin{split}
 \sum_\nu(1+|\nu|_1)^{4m}\theta_{\omega,\nu}^2
 &\le C\left(\norm{v_\omega}_2^2+
               \sum_r\norm{\partial_{u_r}^{2m}v_\omega}_2^2\right)
 \le Ca^2M^{4m-2s},\\
 \sum_\nu\theta_{\omega,\nu}^2&\le Ca^2M^{-2s}.
 \end{split}
\]
H\"older's inequality with exponents determined by $s/(2m)$ now gives
\[
 \sum_\nu(1+|\nu|_1)^{2s}\theta_{\omega,\nu}^2
 \le Ca^2.
\]
As for Legendre series, the powers of $M$ cancel and the coefficients with outcome degree zero vanish. A fixed sufficiently small $a$ therefore puts every alternative in the required degree class, uniformly over the smoothness interval. The preserved $L_2$ distances and the bounds on $g_0$ give the same Hellinger separation and neighboring-sample total-variation bounds as above. The binary-testing argument proves the Chebyshev lower bound.

\textit{4. Bayesian rates.} If $d(j)$ is the degree of the $j$th term, the counting bound gives $j^{1/q}\le1+d(j)$. Hence the degree class of radius $R$ lies in $\Theta_{s/q}(R)$. Apply Theorem~\ref{thm:floor} with a prior radius strictly larger than $R$ to obtain \eqref{eq:all-s-bayes-rate}. For the positive subclass, Proposition~\ref{prop:classical} gives $\rho=2$ for Legendre and $\rho=1$ for Chebyshev. The second assertion of Theorem~\ref{thm:floor} therefore removes the logarithmic factor at the stated smoothness thresholds. Uniformity holds on compact intervals separated from those thresholds, with dependence on the fixed positive lower bound allowed. These Bayesian upper bounds require only the upper bound on $g_0$.
\end{proof}

\section{Proof for independent Gaussian coefficients}\label{app:ordinary}
\begin{proof}[Proof of Theorem~\ref{thm:ordinary}]
Fix $\alpha=\alpha_0$ and define
\[
 \begin{split}
 I_n&=(n/\log n)^{1/(2\alpha+1)},\qquad J_* =\lceil I_n\rceil,
 \qquad\widetilde\eps=I_n^{-\alpha},\\
 T&=n\widetilde\eps^2=I_n\log n,
 \qquad\eps=\widetilde\eps(\log n)^{(1-\zeta)/2}.
 \end{split}
\]
Thus $n\eps^2=T(\log n)^{1-\zeta}$. The constants below may depend on the fixed smoothness endpoints, $R_0,b_0,\overline g,C_K,\rho$, and the prior parameters; they are uniform in $n$, $\alpha$, $p_0$, and $G_0$ in the stated classes. We use only the upper bound on $g_0$.

\textit{1. An approximating event and the denominator.} Fix $0<r\le1$. Conditional on $J=J_*$, consider
\[
 B_n=\{\theta:\norm{\theta-\theta_0^{(J_*)}}_2
                         \le r\widetilde\eps\}.
\]
The total coefficient error is at most $C\widetilde\eps$ on this ball. The truncation bound in Lemma~\ref{lem:interpolation} and the finite-series basis bound give
\[
 \begin{split}
 \norm{h_\theta-h_0}_\infty
 &\le CR_0J_*^{\rho/2-\alpha}
       +\sqrt{C_K}J_*^{\rho/2}r\widetilde\eps\\
 &\le CJ_*^{\rho/2-\alpha}\longrightarrow0
 \end{split}
\]
uniformly because $\alpha\ge\alpha_->\rho/2$. The representations on $B_n$ therefore have a common positive lower bound and a common finite upper bound for all sufficiently large $n$. The last assertion of Lemma~\ref{lem:localKL} yields $K+V\le C\widetilde\eps^2$ and a fixed bound on the absolute log ratio.

For $\theta\in B_n$, $\norm{\theta}_2\le R_0+1$ eventually. Its Gaussian density is consequently at least
\[
 (2\pi\tau^2)^{-J_*/2}
       e^{-(R_0+1)^2/(2\tau^2)}.
\]
Multiply by the volume of the ball. Stirling's bound for the volume of the Euclidean unit ball gives
\[
 -\log\Pi(B_n\mid J_*)\le CJ_*\log n\le CT.
\]
The dimension prior also satisfies $-\log\pi_\zeta(J_*)\le CJ_*[\log(e+J_*)]^\zeta\le CT$. Lemma~\ref{lem:denominator} therefore supplies an event $\mathcal E_n$ with
\begin{equation}\label{eq:ordinaryden}
 P_0^n(\mathcal E_n^c)\le e^{-cT},\qquad
 D_n\ge e^{-C_DT}\quad\text{on }\mathcal E_n.
\end{equation}

\textit{2. Prior probability outside a sieve and its covering number.} For fixed $L_1,L_2>0$, retain parameters satisfying
\[
 J\le\overline J_n=
 \left\lceil\frac{L_1T}{(\log n)^\zeta}\right\rceil,
 \qquad\max_{j\le J}|\theta_j|\le B_n^*=\sqrt{L_2T}.
\]
Let $\mathcal F_n$ be the corresponding set of densities. Splitting the dimension-prior exponent into two equal parts and summing one part gives
\[
 \Pi(J>m)\le C\exp\{-\tfrac\lambda2m[\log(e+m)]^\zeta\}.
\]
The remaining sum is finite because $[\log(e+J)]^\zeta\ge1$. Since $\log\overline J_n\asymp\log n$ uniformly in $\alpha$, choose $L_1$ so that the dimension tail is at most $Ce^{-(C_D+3)T}$. The conditional Gaussian tail bound gives
\[
 \Pi\{J\le\overline J_n,\ \max_{j\le J}|\theta_j|>B_n^*\}
 \le2\overline J_n e^{-L_2T/(2\tau^2)}.
\]
Choose $L_2$ so that the total excluded prior probability is at most $Ce^{-(C_D+2)T}$.

For each retained dimension, cover its coefficient cube by Euclidean balls and apply Proposition~\ref{prop:geometry}. For $M\ge1$,
\begin{align*}
 \log N(M\eps/64,\mathcal F_n,d_{G_0})
 &\le\log(\overline J_n+1)
   +\overline J_n\log\{1+CB_n^*\sqrt{\overline J_n}/\eps\}\\
 &\le C_Hn\eps^2.
\end{align*}
The logarithm is uniformly $O(\log n)$ because its argument is bounded by a fixed power of $n$, and $\overline J_n\log n\le C L_1T(\log n)^{1-\zeta}$. After $L_1,L_2$ are fixed, $C_H$ is independent of $n$, $\alpha$, $p_0$, and $G_0$.

\textit{3. Posterior loss and predictive risk.} Lemma~\ref{lem:tests} gives type-I error at most $e^{(C_H-c_TM^2)n\eps^2}$ and type-II error at most $e^{-c_TM^2n\eps^2}$ for the part of the sieve farther than $M\eps$ from $p_0$. Combine this test with \eqref{eq:ordinaryden} in \eqref{eq:posterior-decomposition}. The excluded prior probability, divided by the denominator lower bound, is at most $Ce^{-2T}$. The integrated type-II term is at most $e^{C_DT-c_TM^2n\eps^2}$. Since $n\eps^2\ge T$, choosing $M$ with $c_TM^2>C_H+C_D+2$ proves \eqref{eq:ordinary-bound}.

Squared Hellinger loss is at most two. Splitting its posterior expectation at $M^2\eps^2$ proves the risk bound, because $e^{-cT}=o(\eps^2)$ uniformly. Convexity gives the predictive risk bound. For each fixed true distribution, the exponential expectation bound is summable in $n$; Markov's inequality and Borel--Cantelli therefore give the stated almost-sure posterior convergence.
\end{proof}

\section{Proof for selection of polynomial terms}\label{app:sparse}
\begin{proof}[Proof of Theorem~\ref{thm:sparse}]
Fix $q_0>0$, put $q_*=\max(q_0,2)$, and consider $n$ with $\eta_n\le1$. Write $T=n\eta_n^2\ge\log n$. The admissible collection is nonempty and finite, so the minimum in \eqref{eq:sparseoracle} is attained. Let $S_*$ be a minimizing subset and put $k_*=|S_*|$. Bessel's inequality and the definition of the oracle criterion give
\begin{equation}\label{eq:oraclecomponents}
 \norm{\theta_{0,S_*}}_2\le R_0,\qquad
 \norm{h_{0,S_*}-h_0}_2\le\eta_n,\qquad
 c_n(S_*)+\log n\le T.
\end{equation}
All constants below are uniform in $n,K_n,G_0,h_0$ under the theorem's restrictions. They may depend on $q_0,\kappa,C_K,\rho,b_0,B_0,R_0,\overline g,A_s$, and $\tau$. No lower bound on $g_0$ is used.

\textit{1. Prior probability near the approximating coefficients.} Choose $D_b>1/2+\kappa\rho/2$ and let $B_n$ be the coefficient ball on $S_*$ with center $\theta_{0,S_*}$ and radius $r_n=n^{-D_b}$. On this ball,
\[
 \norm{h_\theta-h_{0,S_*}}_\infty
 \le\sqrt{C_K}K_n^{\rho/2}r_n\longrightarrow0,
 \qquad
 \norm{h_\theta-h_0}_2\le r_n+\eta_n\le2\eta_n
\]
uniformly. The last inequality uses $\eta_n\ge\sqrt{\log n/n}$. Admissibility gives $\norm{h_{0,S_*}-h_0}_\infty\le b_0/4$; for all sufficiently large $n$, the coefficient perturbation also has supremum norm at most $b_0/4$. Consequently, $b_0/2\le h_\theta\le B_0+b_0/2$ on $B_n$. The bounded-positive-representation assertion of Lemma~\ref{lem:localKL} gives $K+V\le C\eta_n^2$ and a fixed bound on the absolute log ratio. This application does not require an ellipsoid constraint on the coefficients.

If $k_*>0$, integrate the lower bound on the Gaussian density over the coefficient ball. Its center has norm at most $R_0$, its radius is $n^{-D_b}$, and $k_*\le K_n\le n^\kappa$, so
\[
 -\log\Pi(B_n\mid S_*)
 \le C+\{D_bk_*+\tfrac12\kappa k_*\}\log n+Ck_*
 \le C(k_*+1)\log n.
\]
If $k_*=0$, the conditional probability is one. To include the subset probability, use
\begin{equation}\label{eq:kraft}
 \sum_S e^{-c_n(S)}=\sum_{k=0}^{K_n}n^{-k}\le2,
 \qquad
 1\le\sum_S e^{-A_sc_n(S)}\le2,
 \qquad n\ge2.
\end{equation}
The second bound uses $A_s>1$ and $c_n(S)\ge0$. Thus $\pi_n(S_*)\ge\tfrac12e^{-A_sc_n(S_*)}$. Equation~\eqref{eq:oraclecomponents} now shows that the joint approximating event has prior probability at least $e^{-C_1T}$.

Choose the fixed margin in Lemma~\ref{lem:denominator} so that its exponent is at least $(q_*+2)T$. Such a choice is possible because the log-ratio bound is fixed and $u^2/\{2(C_V+Lu/3)\}$ increases without bound as $u\to\infty$. We obtain an event $\mathcal E_n$ with
\begin{equation}\label{eq:sparseden}
 P_0^n(\mathcal E_n^c)\le e^{-(q_*+2)T},\qquad
 D_n\ge e^{-C_DT}\quad\text{on }\mathcal E_n.
\end{equation}

\textit{2. Prior probability outside a sieve and its covering number.} Retain parameters satisfying $c_n(S)\le L_1T$ and $\norm{\theta}_\infty\le\sqrt{L_2T}$, and let $\mathcal F_n$ be their densities. Equation~\eqref{eq:kraft} gives
\[
 \Pi_n\{c_n(S)>L_1T\}\le2e^{-(A_s-1)L_1T}.
\]
The Gaussian tail for the coefficient restriction is bounded by
\[
 2K_ne^{-L_2T/(2\tau^2)}
 \le2e^{-(L_2/(2\tau^2)-\kappa)T},
\]
where $\log K_n\le\kappa\log n\le\kappa T$. Choose $L_1,L_2$ so that the total excluded prior probability is at most $Ce^{-(C_D+q_*+2)T}$.

Again using \eqref{eq:kraft},
\[
 \#\{S:c_n(S)\le L_1T\}\le2e^{L_1T},\qquad
 |S|\le L_1T/\log n
\]
for retained subsets. Cover the coefficient cube on each subset and apply Proposition~\ref{prop:geometry}. For $M\ge1$,
\[
 \begin{split}
 \log N(M\eta_n/64,\mathcal F_n,d_{G_0})
 \le{}&\log2+L_1T\\
 &+\max_{S:c_n(S)\le L_1T}
 |S|\log\{1+C\sqrt{L_2T}\sqrt{|S|}/\eta_n\}
 \le C_HT.
 \end{split}
\]
To verify the last bound, use $\sqrt T/\eta_n=\sqrt n$ and $|S|\le K_n\le n^\kappa$. The logarithm is then $O(\log n)$ uniformly, and $|S|\log n\le L_1T$. The constant $C_H$ is fixed after $L_1,L_2$ have been chosen.

\textit{3. Posterior probability and risk.} Lemma~\ref{lem:tests} gives a test with type-I error at most $e^{(C_H-c_TM^2)T}$ and type-II error at most $e^{-c_TM^2T}$ against densities in the sieve at distance greater than $M\eta_n$ from $p_0$. Use this test and \eqref{eq:sparseden} in \eqref{eq:posterior-decomposition}. The four contributions are the denominator failure, the type-I error, the integrated type-II error, and the excluded prior probability divided by the denominator lower bound. Choosing $M$ sufficiently large makes each contribution at most $e^{-q_*T}$ for all sufficiently large $n$, uniformly in the theorem's classes. The excluded parameters include every subset with $c_n(S)>L_1T$. Thus the same bound proves \eqref{eq:sparse-contraction} with $L=L_1$, since $q_*\ge q_0$.

Splitting squared Hellinger loss at $M^2\eta_n^2$ gives
\[
 \E_0\int d_{G_0}^2(p_\theta,p_0)\,
          \Pi_n(d\theta\mid\cD_n)
 \le M^2\eta_n^2+8e^{-q_*T}
 \le(M^2+8)\eta_n^2.
\]
The last inequality uses $q_*\ge2$, $T\ge\log n$, and $n^{-2}\le\eta_n^2$. Convexity of squared Hellinger loss gives the same upper bound for predictive risk.
\end{proof}

\section{Lower bound for selection of polynomial terms}\label{app:sparselower}
\begin{proof}[Proof of Theorem~\ref{thm:sparse-lower}]
We construct well-separated positive densities whose selected terms are difficult to distinguish. The constants depend only on the fixed covariate-density and basis-growth bounds, and are uniform in $n,k,K$ subject to the theorem's restrictions.

\textit{1. Separated subsets.} Set $m=\lfloor K/k\rfloor\ge16$. Divide $km$ of the first $K$ terms into $k$ groups of $m$ terms. Each vector $v\in\{1,\ldots,m\}^k$ selects one term from every group. Construct a collection $\mathcal V$ by repeatedly selecting a vector and removing every remaining vector within Hamming distance $\lfloor k/4\rfloor$. The selected vectors have pairwise Hamming distance greater than $\lfloor k/4\rfloor$, and therefore at least $k/4$. Each selection removes at most $\sum_{j=0}^{\lfloor k/4\rfloor}\binom kj(m-1)^j$ vectors. Hence
\[
 |\mathcal V|\ge
 \frac{m^k}{\sum_{j=0}^{\lfloor k/4\rfloor}\binom kj(m-1)^j}
 \ge\frac{m^{3k/4}}{2^k},\qquad
 \log|\mathcal V|\ge\tfrac12 k\log m.
\]
The last inequality uses $\log m\ge4\log2$. Choose $c_1\ge8\log2$ in the theorem, so that $\log|\mathcal V|\ge4\log2$.

\textit{2. Positive density alternatives.} Put $\delta^2=a\log m/n$, where $a>0$ will be fixed below. For $v\in\mathcal V$, assign coefficient $\delta$ to its selected terms and zero to all other terms. Define
\[
 h_v=1+\sum_j\theta_{v,j}\phi_j,\qquad p_v=p_{h_v}.
\]
Each representation has outcome integral one. The basis-growth bound yields
\[
 \norm{h_v-1}_\infty
 \le\sqrt{k}\delta\sqrt{C_KK^\rho}.
\]
Since $\log m\le\log(K/k)$, condition \eqref{eq:sparse-local-condition} bounds the square of the right-hand side by $C_Kac_2$. Choose $c_2$ so that this is at most $1/4$. Then $1/2\le h_v\le3/2$, and every alternative belongs to $\mathcal P_{k,K}$.

Let $d_H^{\mathrm{word}}(v,v')$ count the groups in which two vectors select different terms. Orthonormality and the inverse inequality in Proposition~\ref{prop:geometry} give
\[
 \begin{split}
 d_{G_0}^2(p_v,p_{v'})
 &\ge c\norm{\theta_v-\theta_{v'}}_2^2
 =2c\delta^2d_H^{\mathrm{word}}(v,v')\\
 &\ge\tfrac c2 k\delta^2.
 \end{split}
\]
This lower comparison uses $g_0\ge\underline g$. For the reference density $p_*=1$, the bounded-positive-representation part of Lemma~\ref{lem:localKL} and $g_0\le\overline g$ give
\[
 \KL(P_{p_v,G_0}^n,P_{p_*,G_0}^n)
 =nK_{G_0}(p_v,p_*)
 \le Cn\norm{\theta_v}_2^2
 =Ca k\log m.
\]
Choose $a$ so that this is at most $\log|\mathcal V|/8$. Both $a$ and the corresponding $c_2$ can be fixed independently of $n,k,K$.

\textit{3. Testing and estimation risk.} Let $V$ be uniform on $N$ alternatives with sample distributions $P_v$. For any estimator $\widehat V$ and any reference distribution $Q$,
\[
 P(\widehat V\ne V)
 \ge1-\frac{N^{-1}\sum_v\KL(P_v,Q)+\log2}{\log N}.
\]
To verify this form of Fano's inequality, the mutual information between $V$ and the sample is at most $N^{-1}\sum_v\KL(P_v,Q)$. Conditional entropy satisfies $H(V\mid\cD_n)\le\log2+P(\widehat V\ne V)\log N$, by conditioning first on the error indicator. Combine these bounds with $H(V)=\log N$.

Apply the inequality with $N=|\mathcal V|$, $P_v=P_{p_v,G_0}^n$, and $Q=P_{p_*,G_0}^n$. The preceding choices give average error probability at least $1/2$. For any conditional-density estimator, select its nearest alternative in $d_{G_0}$, with fixed tie breaking. An incorrect choice implies that its distance from the true density is at least half the minimum distance between alternatives. Its average squared risk is therefore at least $ck\delta^2$. Finally, $\log\lfloor K/k\rfloor\asymp\log(K/k)$ uniformly for $K/k\ge16$. The supremum risk thus exceeds $ck\log(K/k)/n$. Taking the infimum over estimators proves the theorem.
\end{proof}

\begingroup
\small
\setstretch{1.0}
\setlength{\bibsep}{3pt}

\endgroup

\begin{thebibliography}{99}
\bibitem[Birg\'e(2006)]{Birge2006}
Birg\'e, L. (2006): ``Model Selection via Testing: An Alternative to (Penalized) Maximum Likelihood Estimators,'' \emph{Annales de l'Institut Henri Poincar\'e, Probabilit\'es et Statistiques}, 42, 273--325. doi:10.1016/j.anihpb.2005.04.004.

\bibitem[Gallant and Tauchen(1989)]{GT1989}
Gallant, A. R., and G. Tauchen (1989): ``Seminonparametric Estimation of Conditionally Constrained Heterogeneous Processes: Asset Pricing Applications,'' \emph{Econometrica}, 57, 1091--1120. doi:10.2307/1913624.

\bibitem[Gao and Zhou(2016)]{GaoZhou2016}
Gao, C., and H. H. Zhou (2016): ``Rate Exact Bayesian Adaptation with Modified Block Priors,'' \emph{Annals of Statistics}, 44, 318--345. doi:10.1214/15-AOS1368.

\bibitem[Ghosal, Ghosh, and van der Vaart(2000)]{GGV2000}
Ghosal, S., J. K. Ghosh, and A. W. van der Vaart (2000): ``Convergence Rates of Posterior Distributions,'' \emph{Annals of Statistics}, 28, 500--531. doi:10.1214/aos/1016218228.

\bibitem[Ghosal, Lember, and van der Vaart(2008)]{GLV2008}
Ghosal, S., J. Lember, and A. W. van der Vaart (2008): ``Nonparametric Bayesian Model Selection and Averaging,'' \emph{Electronic Journal of Statistics}, 2, 63--89. doi:10.1214/07-EJS090.

\bibitem[Luo and Luo(2026)]{LuoLuo2026}
Luo, Y., and X. Luo (2026): ``The Minimax Optimal Convergence Rate of Posterior Density in the Weighted Orthogonal Polynomials,'' arXiv:2603.18490. Version accessed 8 September 2026.

\bibitem[NIST DLMF(2026)]{DLMF}
NIST Digital Library of Mathematical Functions (2026): Chapter 18, ``Orthogonal Polynomials,'' Sections 18.3, 18.8, and 18.14. National Institute of Standards and Technology. Accessed 8 September 2026.

\bibitem[Norets and Pati(2017)]{NoretsPati2017}
Norets, A., and D. Pati (2017): ``Adaptive Bayesian Estimation of Conditional Densities,'' \emph{Econometric Theory}, 33, 980--1012. arXiv:1408.5355.

\bibitem[Scricciolo(2006)]{Scricciolo2006}
Scricciolo, C. (2006): ``Convergence Rates for Bayesian Density Estimation of Infinite-Dimensional Exponential Families,'' \emph{Annals of Statistics}, 34, 2897--2920. doi:10.1214/009053606000000911.

\bibitem[Shen and Ghosal(2015)]{ShenGhosal2015}
Shen, W., and S. Ghosal (2015): ``Adaptive Bayesian Procedures Using Random Series Priors,'' \emph{Scandinavian Journal of Statistics}, 42, 1194--1213. doi:10.1111/sjos.12159.

\bibitem[Fan, Yao, and Tong(1996)]{FanYaoTong1996}
Fan, J., Q. Yao, and H. Tong (1996): ``Estimation of Conditional Densities and Sensitivity Measures in Nonlinear Dynamical Systems,'' \emph{Biometrika}, 83, 189--206. doi:10.1093/biomet/83.1.189.

\bibitem[Hall, Racine, and Li(2004)]{HallRacineLi2004}
Hall, P., J. Racine, and Q. Li (2004): ``Cross-Validation and the Estimation of Conditional Probability Densities,'' \emph{Journal of the American Statistical Association}, 99, 1015--1026. doi:10.1198/016214504000000548.

\bibitem[Shen and Ghosal(2016)]{ShenGhosal2016}
Shen, W., and S. Ghosal (2016): ``Adaptive Bayesian Density Regression for High-Dimensional Data,'' \emph{Bernoulli}, 22, 396--420. doi:10.3150/14-BEJ663.

\end{thebibliography}
\end{document}